\documentclass[a4paper, reqno, 10pt]{amsart}

\usepackage{verbatim}
\usepackage[dvips]{color}
\usepackage[usenames,dvipsnames]{xcolor}
\usepackage{amssymb}
\usepackage[active]{srcltx}
\usepackage[colorlinks,pdfpagelabels,pdfstartview = FitH,bookmarksopen = true,bookmarksnumbered = true,linkcolor = NavyBlue,plainpages = false,hypertexnames = false,citecolor = OliveGreen] {hyperref}
\usepackage[italian, textsize=tiny, color=BlueGreen, backgroundcolor=BlueGreen, linecolor=BlueGreen]{todonotes}
\usepackage{mathtools}
\usepackage{empheq}
\usepackage{a4wide}
\usepackage{aligned-overset}

\newcommand\blfootnote[1]{%
\begingroup
\renewcommand\thefootnote{}\footnote{#1}%
\addtocounter{footnote}{-1}%
\endgroup
}

\makeatletter
\@namedef{subjclassname@2020}{\textup{2020} Mathematics Subject Classification}
\makeatother

\allowdisplaybreaks

\title[Singular elliptic measure data problems with irregular obstacles]{Singular elliptic measure data problems with irregular obstacles}

\author[Byun]{Sun-Sig Byun}
\address{Department of Mathematical Sciences and Research Institute of Mathematics,
Seoul National University, Seoul 08826, Republic of Korea}
\email{byun@snu.ac.kr}

\author[Song]{Kyeong Song}
\address{School of Mathematics,
Korea Institute for Advanced Study, Seoul 02455, Republic of Korea}
\email{kyeongsong@kias.re.kr}

\author[Youn]{Yeonghun Youn}
\address{Department of Mathematics,
Yeungnam University, Gyeongsan 38541, Republic of Korea}
\email{yeonghunyoun@yu.ac.kr}

\subjclass[2020]{35B65;  
35J75; 
35J87; 
35R06; 
}

\keywords{Singular $p$-Laplacian; Irregular obstacle; Measure data; Potential estimate}

\newtheorem{theorem}{Theorem}[section]

\newtheorem{lemma}[theorem]{Lemma}

\theoremstyle{definition}
\newtheorem{definition}[theorem]{Definition}
\newtheorem{remark}[theorem]{Remark}

\numberwithin{equation}{section}

\def\eqn#1$$#2$${\begin{equation}\label#1#2\end{equation}}
\def\charfn_#1{{\raise1.2pt\hbox{$\chi_{\kern-1pt\lower3pt\hbox{{$\scriptstyle#1$}}}$}}}

\makeatletter
\newcommand{\pushright}[1]{\ifmeasuring@#1\else\omit\hfill$\displaystyle#1$\fi\ignorespaces}
\newcommand{\pushleft}[1]{\ifmeasuring@#1\else\omit$\displaystyle#1$\hfill\fi\ignorespaces}
\makeatother
\newcommand{\defref}[1]{\hyperref[#1]{Definition}~\ref{#1}}
\newcommand{\thmref}[1]{\hyperref[#1]{Theorem}~\ref{#1}}
\newcommand{\lemref}[1]{\hyperref[#1]{Lemma}~\ref{#1}}

\newcommand{\secref}[1]{\hyperref[#1]{Section}~\ref{#1}}

\DeclareMathOperator*{\osc}{osc}

\DeclareMathOperator*{\data}{\mathtt{data}}

\def\loc{{\operatorname{loc}}}

\newcommand{\dx}{\,dx}
\newcommand{\supp}{{\rm supp}}

\newcommand{\cex}{c_{\mathrm{ex}}}
\newcommand{\csso}{c_{0}}
\newcommand{\cssot}{c_{1}}
\newcommand{\cssoth}{c_{2}}
\newcommand{\cssf}{c_{3}}

\def\mean#1{\mathchoice%
          {\mathop{\kern 0.2em\vrule width 0.6em height 0.69678ex depth -0.58065ex
                  \kern -0.8em \intop}\nolimits_{\kern -0.4em#1}}%
          {\mathop{\kern 0.1em\vrule width 0.5em height 0.69678ex depth -0.60387ex
                  \kern -0.6em \intop}\nolimits_{#1}}%
          {\mathop{\kern 0.1em\vrule width 0.5em height 0.69678ex
              depth -0.60387ex
                  \kern -0.6em \intop}\nolimits_{#1}}%
          {\mathop{\kern 0.1em\vrule width 0.5em height 0.69678ex depth -0.60387ex
                  \kern -0.6em \intop}\nolimits_{#1}}}

\newcommand{\vertiii}[1]{{\left\vert\kern-0.25ex\left\vert\kern-0.25ex\left\vert #1 
			\right\vert\kern-0.25ex\right\vert\kern-0.25ex\right\vert}}
			
\def\avenorm#1{\mathchoice%
          {\mathop{\kern 0.2em\vrule width 0.6em height 0.69678ex depth -0.58065ex
                  \kern -0.545em \|{#1}\|}}%
          {\mathop{\kern 0.1em\vrule width 0.5em height 0.69678ex depth -0.60387ex
                  \kern -0.495em \|{#1}\|}}%
          {\mathop{\kern 0.1em\vrule width 0.5em height 0.69678ex depth -0.60387ex
                  \kern -0.495em \|{#1}\|}}%
          {\mathop{\kern 0.1em\vrule width 0.5em height 0.69678ex depth -0.60387ex
                  \kern -0.495em \|{#1}\|}}}

\newtoks\by
\newtoks\paper
\newtoks\book
\newtoks\jour
\newtoks\yr
\newtoks\pages
\newtoks\vol
\newtoks\publ

\def\ota{{\hbox{\bf ???}}}
\def\cLear{\by=\ota\paper=\ota\book=\ota\jour=\ota\yr=\ota
\pages=\ota\vol=\ota\publ=\ota}
\def\endpaper{\the\by, \textit{\the\paper},
{\the\jour} \textbf{\the\vol} (\the\yr), \the\pages.\cLear}
\def\endbook{\the\by, \textit{\the\book},
\the\publ, \the\yr.\cLear}
\def\endpap{\the\by, \textit{\the\paper}, \the\jour.\cLear}
\def\endproc{\the\by, \textit{\the\paper}, \the\book, \the\publ,
\the\yr, \the\pages.\cLear}

\begin{document}
\maketitle
\begin{abstract}
We investigate elliptic irregular obstacle problems with $p$-growth involving measure data.
Emphasis is on the strongly singular case $1 < p \le 2-1/n$, and we obtain several new comparison estimates to prove gradient potential estimates in an intrinsic form. 
Our approach can be also applied to derive zero-order potential estimates.
\end{abstract}

\blfootnote{This work was supported by National Research Foundation of Korea grant (NRF-2021R1A4A1027378).}

\section{Introduction}
In this paper, we study obstacle problems related to nonlinear elliptic equations of the type
\begin{equation}\label{model}
-\mathrm{div}\,A(Du) = \mu \quad \text{in } \Omega.
\end{equation}
Here $\Omega \subset \mathbb{R}^{n}$ ($n \ge 2$) is a bounded domain and $\mu$ belongs to $\mathcal{M}_{b}(\Omega)$, that is, the space of all signed Borel measures with finite total mass on $\Omega$.
In the following, we extend $\mu$ to $\mathbb{R}^{n}$ by letting $|\mu|(\mathbb{R}^{n}\setminus\Omega) = 0$.
The continuous vector field $A: \mathbb{R}^{n} \to \mathbb{R}^{n}$ is $C^{1}$-regular on $\mathbb{R}^{n}\setminus\{0\}$ and satisfies the following $p$-growth and ellipticity assumptions:
\begin{equation}\label{growth}
\left\{
\begin{aligned}
|A(z)| + |\partial A(z)|(|z|^{2}+s^{2})^{\frac{1}{2}} &\le L(|z|^{2} +s^{2})^{\frac{p-1}{2}},\\
\nu (|z|^{2} + s^{2})^{\frac{p-2}{2}} |\xi|^{2} &\le \partial A(z) \xi \cdot \xi \\
\end{aligned}
\right.
\end{equation}
for every $z , \xi \in \mathbb{R}^{n}$, where $0<\nu  \le L <\infty$ and $s \ge 0$ are fixed constants. Throughout this paper, we assume
\begin{equation}\label{p.bound} 
1 < p \le 2-\frac{1}{n}.
\end{equation}
Roughly speaking, the obstacle problem we are going to consider is \eqref{model} coupled with a unilateral constraint of the form $u \ge \psi$ a.e. in $\Omega$, with $\psi \in W^{1,p}(\Omega)$ being a given obstacle. Note that if $\mu \in W^{-1,p'}(\Omega)$, then our obstacle problem is represented as the following variational inequality:
\begin{equation}\label{opmu}
\int_{\Omega} A(Du) \cdot D(\phi - u)\,dx \ge \int_{\Omega}(\phi - u)\, d\mu  
\end{equation}
for every $\phi \in u + W^{1,p}_{0}(\Omega)$ with $\phi \ge \psi$ a.e. in $\Omega$. Moreover, the existence and uniqueness of a weak solution to \eqref{opmu} are well known consequences of the monotone operator theory \cite{KS00}.
However, when $\mu \notin W^{-1,p'}(\Omega)$, we cannot consider such a variational inequality. 
In this case, a different notion of solutions to the obstacle problem will be given in \defref{def.sol} below.

\subsection{Nonlinear potential estimates}

Pointwise estimates for solutions to nonlinear elliptic measure data problems like \eqref{model} originated from \cite{KM92, KM94}.
More precisely, these papers fundamentally considered $A$-superharmonic functions and corresponding elliptic problems involving nonnegative measures, by employing the maximum principle approach,
to show the necessity part of the Wiener criterion.
Subsequently, in \cite{TW02}, an alternative approach was employed to prove analogous results for subelliptic problems.
Later, in the papers \cite{DM11AJM, KM12JFA}, pointwise estimates were shown for the case of signed Radon measures with finite total mass using perturbation arguments. 
By combining the findings from the aforementioned papers, we can provide the following summary:
if $u$ solves \eqref{model}, and either $p>2-1/n$ or $\mu \ge 0$, 
then there holds
\begin{equation}\label{KM.result}
|u(x_{0})| \le c\mathbf{W}_{1,p}^{\mu}(x_{0},R) + c\mean{B_{R}(x_{0})}(|u|+Rs)\,dx
\end{equation}
whenever $B_{R}(x_{0}) \Subset \Omega$ is a ball and the right-hand side is finite, where 
\begin{equation*}
 \mathbf{W}_{\beta,p}^{\mu}(x_{0},R) \coloneqq \int_{0}^{R}\left[\frac{|\mu|(B_{\rho}(x_{0}))}{\rho^{n-\beta p}}\right]^{\frac{1}{p-1}}\frac{d\rho}{\rho}, \qquad \beta >0,
\end{equation*}
is the nonlinear Wolff potential of $\mu$. Moreover, when both $\mu$ and $u$ are nonnegative in $B_{R}(x_{0})$, we also have the lower bound
\begin{equation}\label{KM.result2}
\mathbf{W}^{\mu}_{1,p}(x_{0},R) \le cu(x_{0}),
\end{equation}
which shows that the estimate \eqref{KM.result} via $\mathbf{W}^{\mu}_{1,p}$ is sharp.
We also refer to \cite{KM18JEMS} for the extension of \eqref{KM.result} to the $p$-Laplace system with measure data, $p>2-1/n$. However, as far as we are concerned, no vectorial analog of \eqref{KM.result2} is available due to the lack of maximum principle.

Later, pointwise estimates were also obtained for the gradient of solutions to \eqref{model}. The first result was proved in \cite{Min11JEMS}, which asserts that pointwise gradient bounds, like those available for the Poisson equation, hold for \eqref{model} in the case $p=2$:
\[ |Du(x_{0})| \le c\mathbf{I}^{\mu}_{1}(x_{0},R) + c\mean{B_{R}(x_{0})}(|Du|+s)\,dx, \]
where \begin{equation*}
 \mathbf{I}_{1}^{\mu}(x_{0},R) \coloneqq \int_{0}^{R}\frac{|\mu|(B_{\rho}(x_{0}))}{\rho^{n-1}}\frac{d\rho}{\rho}
\end{equation*}
is the truncated 1-Riesz potential of $\mu$. 
For the superquadratic case $p>2$, in \cite{DM11AJM} the following Wolff potential estimate
\begin{equation}\label{DM.result}
|Du(x_{0})| \le c\mathbf{W}^{\mu}_{\frac{1}{p},p}(x_{0},R) + c\mean{B_{R}(x_{0})}(|Du|+s)\,dx
\end{equation}
was proved. See also \cite{KM12JFA} for ``universal'' potential estimates that interpolate \eqref{KM.result} and \eqref{DM.result}. 
Surprisingly, in contrast with the zero-order estimate \eqref{KM.result}, it was proved in \cite{DM10JFA,KM13ARMA} that pointwise gradient estimates via Riesz potentials hold for nonlinear, possibly degenerate equations like \eqref{model}.
More precisely, we have the following: if $u$ solves \eqref{model} under assumptions \eqref{growth} with
\begin{equation}\label{p.sola}
p > 2-\frac{1}{n},
\end{equation}
then it holds that
\begin{equation}\label{riesz-intro}
|Du(x_{0})| \le c[\mathbf{I}^{\mu}_{1}(x_{0},R)]^{\frac{1}{p-1}} + c\mean{B_{R}(x_{0})}(|Du|+s)\,dx,
\end{equation}
whenever $B_{R}(x_{0}) \Subset \Omega$ and the right-hand side is finite.
Moreover, \eqref{riesz-intro} improves \eqref{DM.result} when $p>2$.
Note that, in light of \eqref{growth}, estimate \eqref{riesz-intro} can be rephrased as
\[ |A(Du(x_{0}))| \le c\mathbf{I}^{\mu}_{1}(x_{0},R) + c\mean{B_{R}(x_{0})}|A(Du)|\,dx. \]
We also remark that the results in \cite{DM10JFA,DM11AJM,KM12JFA,KM13ARMA} are concerned with SOLA (Solutions Obtained as Limits of Approximations) introduced in \cite{BG89}, for which the lower bound \eqref{p.sola} is indispensable; see also the discussions after \defref{def.sol} below.

Estimate \eqref{riesz-intro}, known to be the sharp gradient potential estimate for $p$-Laplacian type equations, was further extended to elliptic equations with nonstandard growth \cite{Ba15, BaHa14, BY17, BY18} and parabolic $p$-Laplacian type equations \cite{KM13Pisa, KM14ARMA} with $p>2-1/(n+1)$.
Later in \cite{KM18JEMS}, estimate \eqref{riesz-intro} was also established for measure data systems involving the $p$-Laplacian, $p \ge 2$. Additionally, in the case when the data $\mu$ possesses sufficient regularity to guarantee the existence of weak solutions, it is possible to derive Riesz potential type estimates for elliptic systems without a quasi-diagonal structure in the context of partial regularity, see \cite{BY19,D22,DS23JFA,KM16JEP}.

In the recent papers \cite{DZ,NP20JFA,NP23ARMA}, potential estimates for \eqref{model} were investigated for the range \eqref{p.bound}, where different notions of solutions, such as renormalized solutions or approximable solutions, should be considered.
We refer to the recent papers \cite{DMOP99, CM17NA} for more details about each notion of solutions.
The papers \cite{DZ,NP20JFA,NP23ARMA} proposed new methods in obtaining comparison estimates, which address the difficulties coming from the lack of integrability of $Du$ and the failure of Sobolev-Poincar\'e type inequalities.
In these papers, such difficulties are overcome by initially establishing Marcinkiewicz type estimates and then proving new reverse H\"older type estimates. 
Furthermore, a modified excess functional in the form of \eqref{mod.exs} below was employed.

\subsection{Main results}
Here we describe the formulation of our obstacle problem, $OP(\psi;\mu)$, and the concept of solutions used in this paper. As mentioned above, since $\mu$ does not in general belong to $W^{-1,p'}(\Omega)$, the variational inequality \eqref{opmu} is not available for $OP(\psi;\mu)$. 
In this paper, we consider \textit{limits of approximating solutions} introduced in \cite{Sch12JFA}. 
For other several notions of solutions, see \cite[Section~1.1]{Sch12JFA} and related references therein.

For each $k>0$, we consider the truncation operator $T_{k}:\mathbb{R} \rightarrow \mathbb{R}$ defined by
\begin{equation}\label{truncation.op}
T_{k}(t) \coloneqq \min\{k,\max\{t,-k\}\}, \qquad t \in \mathbb{R}. 
\end{equation}
Given a boundary data $g \in W^{1,p}(\Omega)$, we set
\[ \mathcal{T}^{1,p}_{g}(\Omega) \coloneqq \left\{ u: \Omega \rightarrow \mathbb{R} \mid T_{k}(u-g) \in W^{1,p}_{0}(\Omega) \text{ for every } k>0 \right\}. \]
It is well known that for any $u \in \mathcal{T}_{g}^{1,p}(\Omega)$, there exists a unique measurable map $Z_u: \Omega \to \mathbb{R}^n$ satisfying
\[ D[T_k(u)] = \chi_{\{|u| < k\}} Z_u \qquad \text{a.e. in }\Omega \]
for every $k>0$, see \cite[Lemma~2.1]{BBGGPV1995}.
If $u \in \mathcal{T}_{g}^{1,p}(\Omega) \cap W^{1,1}(\Omega)$, then $Z_u$ coincides with the weak derivative $Du$ of $u$. 
In this paper, we denote $Z_u$ by $Du$ for any $u \in \mathcal{T}^{1,p}_{g}(\Omega)$.

\begin{definition}\label{def.sol}
Suppose that an obstacle $\psi \in W^{1,p}(\Omega)$, measure data $\mu \in \mathcal{M}_{b}(\Omega)$ and boundary data $g \in W^{1,p}(\Omega)$ with $(\psi-g)_{+} \in W^{1,p}_{0}(\Omega)$ are given.
We say that a function $u \in \mathcal{T}^{1,p}_{g}(\Omega)$ with $u \ge \psi$ a.e. in $\Omega$ is a limit of approximating solutions to the obstacle problem $OP(\psi;\mu)$ under assumptions \eqref{growth} with $p>1$, if there exist a sequence of functions $ \{\mu_{k}\} \subset W^{-1,p'}(\Omega)\cap L^{1}(\Omega)$ with 
\begin{equation}\label{muk.conv}
\left\{
\begin{aligned}
&\mu_{k} \overset{\ast}{\rightharpoonup} \mu \;\; \textrm{in } \mathcal{M}_{b}(\Omega), \\
&\limsup_{k\rightarrow\infty} |\mu_{k}|(B) \le |\mu|(\bar{B}) \quad \textrm{for every ball }B \subset \mathbb{R}^{n}
\end{aligned}
\right.
\end{equation}
and weak solutions $ u_{k} \in g + W^{1,p}_{0} (\Omega)$ with $u_{k} \ge \psi $ a.e. in $\Omega$ to the variational inequalities 
\begin{equation*}
\int_{\Omega} A(Du_{k})\cdot D(\phi - u_{k}) \,dx \ge \int_{\Omega}(\phi - u_{k})\, d\mu_{k}  
\end{equation*}
for every $ \phi \in u_{k} + W^{1,p}_{0}(\Omega)$ with $\phi \ge \psi$ a.e. in $\Omega$,
such that
\begin{equation}\label{uk.conv}
\left\{ \begin{aligned}
& u_k \to u  && \text{a.e. in }\Omega,\\
& \int_{\Omega} |u_k - u|^{\gamma} \,dx \to 0 && \text{for every } 0< \gamma < \frac{n(p-1)}{n-p}, \\
& \int_{\Omega} |Du_k - Du|^{q} \,dx \to 0 && \text{for every } 0< q < \frac{n(p-1)}{n-1}.
\end{aligned} \right. 
\end{equation}
\end{definition}
The existence of limits of approximating solutions to $OP(\psi;\mu)$ was proved in \cite{Sch12JFA} by extending the classical approach in \cite{BG89}; see also \cite{SY} for a uniqueness result in the case $\mu \in L^{1}(\Omega)$. 
Now it is easy to see the role of \eqref{p.sola}:
\[ p > 2-\frac{1}{n} \;\; \Longleftrightarrow \;\; \frac{n(p-1)}{n-1}>1. \]
We indeed have $u \in W^{1,1}(\Omega)$ if and only if \eqref{p.sola} is in force. 
Note that, while the convergence property \eqref{uk.conv} is very similar as in the case of SOLA, limits of approximating solutions can be defined for the range \eqref{p.bound} as well.
This is because we do not require $u$ itself to satisfy a distributional formulation.

\subsubsection{Gradient potential estimates}
Gradient potential estimates for $OP(\psi;\mu)$ in the range \eqref{p.sola} were first obtained in \cite{Sch12JFA}, under the assumption that 
\begin{equation*}
\psi \in W^{1,p}(\Omega)\cap W^{2,1}(\Omega) \quad \text{satisfies} \quad \mathcal{D}\Psi \coloneqq \mathrm{div}\,A(D\psi) \in L^{1}(\Omega).
\end{equation*}
Such a higher regularity assumption allows one to apply the methods in \cite{DM10JFA,DM11AJM} to $OP(\psi;\mu)$, treating the obstacle and the measure in the same way. Indeed, the main estimates in \cite{Sch12JFA} involve Wolff potentials (when $p>2$) and Riesz potentials (when $2-1/n < p \le 2$) of $\mu$ and $\mathcal{D}\Psi$. We also refer to \cite{BCP21,BSY2} for integrability and differentiability results for elliptic double obstacle problems with measure data, under similar assumptions on the double obstacles.

In the recent paper \cite{BSY}, a new form of gradient potential estimates for $OP(\psi;\mu)$ was proved under assumptions \eqref{growth} and \eqref{p.sola}, without any higher regularity assumptions on the obstacle. Moreover, Wolff potentials of $\mu$ appearing in \cite[Theorem 4.3]{Sch12JFA} were replaced by Riesz potentials:
\begin{align*} 
|Du(x_{0})|^{p-1} & \le c\mathbf{I}^{\mu}_{1}(x_{0},R) + c\left[\int_{0}^{R}\left(\mean{B_{\rho}(x_{0})}\varphi^{*}(|A(D\psi)-(A(D\psi))_{B_{\rho}}|)\,dx\right)^{\frac{1}{m}}\frac{d\rho}{\rho}\right]^{\frac{m}{p'}}\\
& \quad + c\mean{B_{R}(x_{0})}(|Du|+s)^{p-1}\,dx,
\end{align*}
where $m \coloneqq \max\{p',2\}$, and the function $\varphi^{*}(\cdot)$ is defined in \eqref{phistar} below. 
The approach in \cite{BSY} is based on an intrinsic linearization technique motivated from those in \cite{AKM18,BCDKS18} (see also \cite{BDW20, DKS12}), which enables one to treat both measure data and irregular obstacles simultaneously. We also note that all the estimates were actually formulated in terms of the natural quantity $A(Du)$.

In this paper, we extend the gradient potential estimate in \cite[Theorem~1.2]{BSY} to the range \eqref{p.bound}, as mentioned in \cite{BSY}. 
To this aim, we first extend the approaches in \cite{NP23ME,PS22} to the setting of obstacle problems, by employing new test functions, to establish comparison estimates for $Du$.
We then apply an analog of the alternative scheme in \cite{BSY} to linearize such estimates, which gives an intrinsic form of estimates for $A(Du)$.
Note that, while $Du$ need not be an $L^{1}$-function, we have $A(Du) \in L^{1}(\Omega)$ by \eqref{uk.conv}. Here we set the exponent
\begin{equation}\label{def.kappa}
\kappa \coloneqq \frac{(p-1)^{2}}{2}.
\end{equation}

\begin{theorem}\label{pointwise.est}
Let $u \in \mathcal{T}^{1,p}_{g}(\Omega)$ be a limit of approximating solutions to the problem $OP(\psi;\mu)$ under assumptions \eqref{growth} and \eqref{p.bound}. Then there exists a constant $c \equiv c(n,p,\nu,L)$ such that the pointwise estimate
\begin{align*}
|A(Du)(x_{0})| & \le c\mathbf{I}^{\mu}_{1}(x_{0},2R) + c\int_{0}^{2R}\left(\mean{B_{\rho}(x_{0})}\varphi^{*}(|A(D\psi)-(A(D\psi))_{B_{\rho}(x_{0})}|)\,dx\right)^{\frac{1}{p'}}\frac{d\rho}{\rho} \\
& \quad + c\left(\mean{B_{2R}(x_{0})}|A(Du)|^{\kappa}\,dx\right)^{\frac{1}{\kappa}}
\end{align*}
holds whenever $B_{2R}(x_{0}) \subset \Omega$ and $x_{0}\in\Omega$ is a Lebesgue point of $A(Du)$.
\end{theorem}

The above theorem can be actually obtained as a corollary of a more general result, which we state as follows. See \eqref{mod.exs} below for the definition of $\mathcal{P}_{\kappa,B_{\rho}(x_{0})}(\cdot)$.

\begin{theorem}\label{mainthm.1}
Let $u \in \mathcal{T}^{1,p}_{g}(\Omega)$ be a limit of approximating solutions to the problem $OP(\psi;\mu)$ under assumptions \eqref{growth} and \eqref{p.bound}.
\begin{itemize}
\item If
\begin{equation}\label{mainthm.1.asmp1}
\lim_{\rho \to 0} \left[\frac{|\mu|(B_{\rho}(x_{0}))}{\rho^{n-1}}
+ \left( \mean{B_{\rho}(x_{0})} \varphi^{*}(|A(D\psi)-(A(D\psi))_{B_{\rho}(x_{0})}|) \,dx \right)^{\frac{1}{p'}} \right] = 0
\end{equation}
holds for a point $x_{0}\in\Omega$, then 
\begin{equation}\label{vmo.x0}
\lim_{\rho \to 0} \mean{B_{\rho}(x_0)}|A(Du) - \mathcal{P}_{\kappa,B_{\rho}(x_0)}(A(Du))|^{\kappa} \dx = 0.
\end{equation}
\item If
\begin{equation}\label{mainthm.1.asmp2}
\mathbf{I}_{1}^{\mu}(x_0,2R) + \int_{0}^{2R} \left( \mean{B_{\rho}(x_{0})} \varphi^{*}(|A(D\psi)-(A(D\psi))_{B_{\rho}(x_{0})}|)\,dx  \right)^{\frac{1}{p'}}\frac{d\rho}{\rho} < \infty
\end{equation}
holds for a ball $B_{2R}(x_{0})\subset \Omega$, then the limit
\begin{equation}\label{Lebesgue.pt} 
A_{0} \coloneqq \lim_{\rho\rightarrow0}\mathcal{P}_{\kappa,B_{\rho}(x_{0})}(A(Du))
\end{equation}
exists.
Moreover, the estimate
\begin{align}\label{mainest.1}
\lefteqn{ |A_{0} - \mathcal{P}_{\kappa,B_{2R}(x_0)}(A(Du))| } \nonumber \\
& \le c\left(\mean{B_{2R}(x_{0})}|A(Du) - \mathcal{P}_{\kappa,B_{2R}(x_0)}(A(Du))|^{\kappa}\,dx\right)^{\frac{1}{\kappa}} \nonumber \\
& \quad + c \int_{0}^{2R}\left(\mean{B_{\rho}(x_{0})}\varphi^{*}(|A(D\psi)-(A(D\psi))_{B_{\rho}(x_{0})}|)\,dx \right)^{\frac{1}{p'}}\frac{d\rho}{\rho}
\end{align}
holds for a constant $c \equiv c(n,p,\nu,L)$.
\item Finally, if $x_{0}$ is a Lebesgue point of $A(Du)$, then the limit $A_{0}$ defined in \eqref{Lebesgue.pt} is equal to $A(Du)(x_{0})$.
\end{itemize}
\end{theorem}

\begin{remark}
In the proof of \thmref{mainthm.1}, we can also obtain the following $C^{1}$-regularity criterion (see for instance \cite[Theorem~1]{DM10CV} and \cite[Theorem~4]{KM13ARMA}): if $\mu \in L(n,1)$ locally in $\Omega$ and $A(D\psi)$ has Dini mean oscillation, which means that
\[ \int_{0}[\omega(\rho)]^{\frac{1}{p'}}\frac{d\rho}{\rho} < \infty, \quad \text{where} \quad  \omega(\rho) \coloneqq \sup_{y \in \Omega}\mean{B_{\rho}(y)}\varphi^{*}(|A(D\psi)-(A(D\psi))_{B_{\rho}(y)}|)\,dx, \]
then $Du$ is continuous in $\Omega$.
We also refer to \cite{KM14CV} for a different proof that avoids potentials.
\end{remark}

\subsubsection{Zero-order potential estimates}
We can also obtain potential estimates for $u$, which extend the results in \cite{Sch12PM} to the case \eqref{p.bound}. For simplicity, we only state an analog of \thmref{pointwise.est}.
\begin{theorem}
Let $u \in \mathcal{T}^{1,p}_{g}(\Omega)$ be a limit of approximating solutions to $OP(\psi;\mu)$, with the Carath\'eodory vector field $A:\Omega\times\mathbb{R}^{n}\rightarrow\mathbb{R}^{n}$ satisfying
\begin{equation*}
\left\{
\begin{aligned}
|A(x,z)| & \le L(|z|^{2}+s^{2})^{\frac{p-1}{2}} \\
\nu(|z_{1}|^{2}+|z_{2}|^{2}+s^{2})^{\frac{p-2}{2}}|z_{1}-z_{2}|^{2} & \le (A(x,z_{1})-A(x,z_{2}))\cdot(z_{1}-z_{2})
\end{aligned}
\right.
\end{equation*}
for every $z,z_{1},z_{2} \in \mathbb{R}^{n}$ and a.e. $x\in\Omega$.
Assume that $p$ satisfies \eqref{p.bound}.
Then there exists a constant $c\equiv c(n,p,\nu,L)$ such that the pointwise estimate
\begin{align*}
|u(x_{0})| & \le c\mathbf{W}^{\mu}_{1,p}(x_{0},2R) + c\int_{0}^{2R}\left[\rho^{p}\mean{B_{\rho}(x_{0})}(|D\psi|+s)^{p}\,dx\right]^{\frac{1}{p}}\frac{d\rho}{\rho} \\
& \quad + c\left(\mean{B_{2R}(x_{0})}(|u|+Rs)^{\kappa}\,dx\right)^{\frac{1}{\kappa}}
\end{align*}
holds whenever $B_{2R}(x_{0})\subset\Omega$, for a.e. $x_{0}\in\Omega$.
\end{theorem}

\begin{remark}
Note that comparison estimates between homogeneous obstacle problems and obstacle-free problems in \cite[Section~3.2]{Sch12PM} are valid for every $p>1$, since they are concerned with weak solutions. 
Thus, once we have the comparison estimate given in \lemref{Du-Dw1} below, the above theorem can be proved by the arguments in \cite[Section~4]{Sch12PM}, see also \cite{CS18,DM11AJM}. 
Moreover, the $C^{0}$-regularity criterion in \cite[Theorem~4.6]{Sch12PM} can be also extended to the range \eqref{p.bound}:
\[ \mu \in L\left(\frac{n}{p},\frac{1}{p-1}\right), \, D\psi \in L(n,1) \text{ locally in }\Omega \;\; \Longrightarrow \;\; u \text{ is continuous in } \Omega. \]
\end{remark}

The organization of this paper is as follows. In the next section, we introduce some notations and preliminary materials. \secref{sec.reg.homo} is devoted to regularity results for homogeneous obstacle problems and homogeneous equations.
In \secref{sec.comparison.est} and \secref{sec.lin.comp}, we establish several comparison estimates between \eqref{opmu} and the corresponding reference problems. Finally, in \secref{sec.pf.thm1} we prove \thmref{mainthm.1}.

\section{Preliminaries} 
\subsection{Notation}
We denote by $c$ a general constant greater than or equal to one; special occurrences will be denoted by $c_{*},c_{0}$, etc. The value of $c$ may vary from line to line. Specific dependencies of constants are denoted by parentheses, and we use the abbreviation
\begin{equation*}
\data \coloneqq (n, p,\nu,L).
\end{equation*}
Additionally, we write $a \approx b$ if there is a constant $c \geq 1$ depending only on $\data$ such that $c^{-1} a \leq b \leq c a$.
For any $q > 1$, we denote its H\"older conjugate exponent by $q' \coloneqq q/(q-1)$.
As usual, with $x = (x_{1},\ldots,x_{n}) \in \mathbb{R}^{n}$, we denote by
\[ B_{r}(x) \coloneqq \left\{ y\in\mathbb{R}^{n} : |y-x| < r \right\}  \quad \text{and} \quad 
 Q_{r}(x) \coloneqq \left\{ y\in\mathbb{R}^{n} : \sup_{1\le i \le n}|y_{i}-x_{i}| < r \right\} \]
the open ball and cube, respectively, with center $x$ and ``radius'' $r>0$.
If there is no confusion, we omit the centers and simply write $B_{r} \equiv B_{r}(x)$ and $Q_{r} \equiv Q_{r}(x)$.
Also, given a ball $B$ and a cube $Q$, we denote by $\gamma B$ and $\gamma Q$ the concentric ball and cube, respectively, with radius magnified by a factor $\gamma>0$. Unless otherwise stated, different balls or cubes in the same context are concentric. 
Moreover, when considering cubes, we identify $\mathbb{R}^{n} \equiv \mathbb{R}^{n-1}\times \mathbb{R}$, denoting each element as $x = (x',x_{n})$. We accordingly denote
\[ Q_{r}'(x') \coloneqq \left\{ y' \in \mathbb{R}^{n-1}:\sup_{1\le i \le n-1}|y_{i} - x_{i}| < r \right\}  \] 
so that $Q_{r}(x) = Q_{r}'(x') \times (x_{n}-r,x_{n}+r)$.

The ($n$-dimensional) Lebesgue measure of a measurable set $S \subset \mathbb{R}^n$ is denoted by $|S|$. For an integrable map $f:S \to \mathbb{R}^{k}$, with $k \ge 1$ and $0<|S|<\infty$, we write
\begin{equation*}
(f)_{S} \coloneqq \mean{S}f\,dx  \coloneqq \frac{1}{|S|} \int_{S}f\,dx
\end{equation*}
to mean the integral average of $f$ over $S$.
The oscillation of $f$ on $S$ is defined by
\begin{equation*}
\osc_{S}f \coloneqq \sup_{x,y\in S}|f(x)-f(y)|.
\end{equation*}

We shall identify a function $\mu \in L^{1}(\Omega)$ with a signed measure, by denoting
\begin{equation*}
|\mu|(S) = \int_{S}|\mu|\,dx \quad \textrm{for each measurable subset } S \subseteq \Omega,
\end{equation*}
and thereby identify $L^{1}(\Omega)$ with a subset of $\mathcal{M}_{b}(\Omega)$.

We use the following short notations for the admissible sets of the problem $OP(\psi;\mu)$: given an open set $\mathcal{O} \subseteq \Omega$ and a function $g \in W^{1,p}(\mathcal{O})$ with $g \ge \psi$ a.e. in $\mathcal{O}$, we denote
\begin{align*}
\mathcal{A}_{\psi}(\mathcal{O}) &\coloneqq \left\{ \phi \in W^{1,p}(\mathcal{O}): \phi \ge \psi \textrm{ a.e. in } \mathcal{O} \right\}, \\
\mathcal{A}_{\psi}^{g}(\mathcal{O}) &\coloneqq \left\{ \phi \in g + W^{1,p}_{0}(\mathcal{O}): \phi \ge \psi \textrm{ a.e. in } \mathcal{O} \right\}.
\end{align*}

\subsection{Basic properties of the vector fields $V(\cdot)$ and $A(\cdot)$}
Recall that the ellipticity assumption in \eqref{growth} implies the following monotonicity property: 
\begin{equation*}
(A(z_{1})-A(z_{2}))\cdot(z_{1}-z_{2}) \approx (|z_{1}|^{2} + |z_{2}|^{2} + s^{2})^\frac{p-2}{2} |z_{1}-z_{2}|^{2}
\end{equation*}
for any $z_{1},z_{2} \in \mathbb{R}^{n}$.

We now consider the auxiliary vector field $V \equiv V_{s}:\mathbb{R}^{n}\rightarrow\mathbb{R}^{n}$ defined by 
\begin{equation*}
V(z) \equiv V_{s}(z) \coloneqq (|z|^{2}+s^{2})^{\frac{p-2}{4}}z,\qquad z \in \mathbb{R}^{n}.
\end{equation*}
It is well known that 
\begin{equation}\label{V}
   |V(z_{1})-V(z_{2})| \approx (|z_{1}|^2 + |z_{2}|^2 + s^{2})^\frac{p-2}{4} |z_{1}-z_{2}|
\end{equation}
holds for any $z_{1},z_{2} \in \mathbb{R}^n$, 
where the implicit constant depends only on $p$. Specifically, in view of \eqref{V}, the vector field $V(\cdot)$ is naturally linked to the monotonicity of $A(\cdot)$. Namely, for any $z_{1},z_{2}\in \mathbb{R}^n$ there holds
\begin{equation}\label{mono.V}
(A(z_{1})-A(z_{2}))\cdot(z_{1}-z_{2}) \approx |V(z_{1})-V(z_{2})|^{2}.
\end{equation}
We further recall some properties of the vector field $A(\cdot)$; see \cite[Lemma 2.1]{AKM18}.
\begin{lemma}
The following inequalities hold for every choice of $z, z_{1}, z_{2} \in \mathbb{R}^{n}$:
\begin{equation}
\begin{aligned}
|A(z)| + s^{p-1} & \approx |z|^{p-1} + s^{p-1} \approx (|z|+s)^{p-1}, \\
|A(z_{1})-A(z_{2})| & \approx (|z_{1}|^{2}+|z_{2}|^{2}+s^{2})^{\frac{p-2}{2}}|z_{1}-z_{2}|. \label{a.diff}
\end{aligned}
\end{equation}
In particular, $A(\cdot)$ is a locally bi-Lipschitz bijection, and it holds that
\begin{equation*}
|A(z_{1})-A(z_{2})| \leq c |z_{1}-z_{2}|^{p-1} \quad \textrm{when } 1 < p \le 2,
\end{equation*}
for some $c = c(\data)$.
\end{lemma}

We also recall several properties of shifted power functions which are useful in dealing with divergence type data.
For a comprehensive introduction, see \cite{BDW20,BCDKS18,DFTW,DKS12} and references therein. 
For each $a \geq 0$, we define the function $\varphi_{a}(\cdot)$ by
\begin{equation*}
\varphi_{a}(t) \coloneqq (a+s+t)^{p-2}t^{2}, \qquad t \ge 0.
\end{equation*}
We simply denote $\varphi_{0} \equiv \varphi$. 
Then $\varphi_{a}(\cdot)$ is an $N$-function, i.e., it has a right continuous, non-decreasing derivative $\varphi_{a}'(\cdot)$ which satisfies $\varphi_{a}'(0) = 0$ and $\varphi_{a}'(t) > 0$ for $t>0$.
Moreover, a direct calculation shows that
\begin{equation}\label{shifted.ftn.o}
\min\{p-1,1\} \le \frac{t\varphi_{a}''(t)}{\varphi_{a}'(t)} \le \max\{p-1,1\} \quad \text{and} \quad 
\min\{p,2\} \le \frac{t\varphi_{a}'(t)}{\varphi_{a}(t)} \le \max\{p,2\}
\end{equation}
hold for any $t \ge 0$.
In particular, $\eqref{shifted.ftn.o}_{2}$ implies that the family $\{\varphi_{a}\}_{a\ge0}$ satisfies the $\Delta_{2}$ and $\nabla_{2}$ conditions uniformly in $a$, i.e., $\varphi_{a}(2t) \approx \varphi_{a}(t)$ uniformly in $a,t\ge0$.
Accordingly, we can consider the complementary $N$-function of $\varphi_{a}(\cdot)$ which is defined by
\begin{equation}\label{phistar}
(\varphi_{a})^{*}(t) \coloneqq \sup_{\tau \ge 0}\left(\tau t - \varphi_{a}(\tau)\right), \qquad t \ge 0.
\end{equation}
We indeed have
\begin{equation*}
(\varphi_{a})^{*}(t) \approx ((a+s)^{p-1} + t)^{p'-2}t^{2}, \qquad t \ge 0.
\end{equation*}
Shifted $N$-functions are especially useful when describing the monotonicity property of $A(\cdot)$:
\begin{equation}\label{mono.shift}
\begin{aligned}
(A(z_{1})-A(z_{2}))\cdot(z_{1}-z_{2}) & \approx |V(z_{1})-V(z_{2})|^{2} \\
& \approx \varphi_{|z_{1}|}(|z_{1}-z_{2}|) \approx (\varphi_{|z_{1}|})^{*}(|A(z_{1})-A(z_{2})|).
\end{aligned}
\end{equation}
We also note the following ``shift change formula''
\begin{equation}\label{shift.change}
\begin{aligned}
\varphi_{|z_{1}|}(t) & \le c\varepsilon^{1-\max\{p',2\}}\varphi_{|z_{2}|}(t) + \varepsilon|V(z_{1})-V(z_{2})|^{2}, \\
(\varphi_{|z_{1}|})^{*}(t) & \le c\varepsilon^{1-\max\{p,2\}}(\varphi_{|z_{2}|})^{*}(t) + \varepsilon|V(z_{1})-V(z_{2})|^{2},
\end{aligned}
\end{equation}
valid for any $z_{1},z_{2} \in \mathbb{R}^{n}$, $\varepsilon \in (0,1]$ and $t \ge 0$.

\subsection{A modified excess functional}
We recall the following inequality: if $S \subset \mathbb{R}^{n}$ is a measurable set with $0<|S|<\infty$ and $f\in L^{q}(S;\mathbb{R}^{k})$ for some $q \in [1,\infty)$, then we have
\begin{equation}\label{mean.min}
\left(\mean{S}|f-(f)_{S}|^{q}\,dx\right)^{\frac{1}{q}} \le 2\left(\mean{S}|f-z_{0}|^{q}\,dx\right)^{\frac{1}{q}} \qquad \forall \; z_{0} \in \mathbb{R}^{k}.
\end{equation}
The quantity on the left-hand side of \eqref{mean.min} is often called an excess functional. Such a quantity naturally appears in various subjects including Campanato's theory.

In view of \eqref{mean.min}, we consider, this time for any $q \in (0,\infty)$, the following ``modified excess functional'' 
\[ \inf_{z_{0} \in \mathbb{R}^{k}}\left(\mean{S}|f-z_{0}|^{q}\,dx \right)^{\frac{1}{q}}. \]
Then there exists a vector $\mathcal{P}_{q,S}(f) \in \mathbb{R}^{k}$ such that
\begin{equation}\label{mod.exs}
\left(\mean{S}|f-\mathcal{P}_{q,S}(f)|^{q}\,dx\right)^{\frac{1}{q}} = \inf_{z_{0}\in\mathbb{R}^{k}}\left(\mean{S}|f-z_{0}|^{q}\,dx\right)^{\frac{1}{q}}.
\end{equation}
It is well known that $\mathcal{P}_{2,S}(f) = (f)_{S}$. However, even if $f \in L^{1}(S)$,  \eqref{mean.min} may fail for $q<1$, see \cite[Section~\uppercase\expandafter{\romannumeral3}.A]{DJL92}. We also note that $\mathcal{P}_{q,S}(f)$ is not in general uniquely determined, for instance, when $q<1$. In this paper, when referring to $\mathcal{P}_{q,S}(f)$, we take any possible value of it.
We note that
\begin{align}\label{av.min}
|\mathcal{P}_{q,S}(f) - z_{0}| &= \left(\mean{S}|\mathcal{P}_{q,S}(f) - z_{0}|^{q}\,dx\right)^{\frac{1}{q}} \nonumber \\
& \le c\left(\mean{S}|\mathcal{P}_{q,S}(f)-f|^{q}\,dx\right)^{\frac{1}{q}} + c\left(\mean{S}|f-z_{0}|^{q}\,dx\right)^{\frac{1}{q}} \nonumber \\
& \le c\left(\mean{S}|f-z_{0}|^{q}\,dx\right)^{\frac{1}{q}}
\end{align}
holds for a constant $c\equiv c(q)$, whenever $z_{0} \in \mathbb{R}^{k}$.
Moreover, the following analog of Lebesgue's differentiation theorem holds (see for instance \cite[Lemma~4.1]{DVS84}): If $f \in L^{q}_{\loc}(\mathbb{R}^{n})$, then
\[ \lim_{\rho \rightarrow 0}\mathcal{P}_{q,Q_{\rho}(x_{0})}(f) = f(x_{0}) \qquad \text{for a.e. } x_{0} \in \mathbb{R}^{n}. \]

\section{Regularity for reference problems}\label{sec.reg.homo}

We first note a reverse H\"older type inequality for the following homogeneous obstacle problem:
\begin{equation}\label{homogeneous.obstacle}
\left\{
\begin{aligned}
\int_{\Omega} A(Dw_{1})\cdot D(\phi - w_{1})\,dx &\ge 0 \qquad \forall\;\phi \in \mathcal{A}^{w_{1}}_{\psi}(\Omega)\\
w_{1} & \ge \psi \qquad \textrm{a.e. in } \Omega.
\end{aligned}
\right.
\end{equation}

\begin{lemma}\label{reverse.holder.w1}
Let $w_{1} \in \mathcal{A}_{\psi}(\Omega)$ be a weak solution to \eqref{homogeneous.obstacle} under assumptions \eqref{growth} with $p>1$. Then, with $\kappa$ defined in \eqref{def.kappa}, there exists a constant $c \equiv c(\data)$ such that
\begin{align*}
\lefteqn{\mean{Q}|V(Dw_{1})-V(z_{0})|^{2}\,dx} \nonumber \\
& \le c(\varphi_{|z_{0}|})^{*}\left[\left(\mean{2Q}|A(Dw_{1})-A(z_{0})|^{\kappa}\,dx\right)^{\frac{1}{\kappa}}\right] + c\mean{2Q}(\varphi_{|z_{0}|})^{*}(|A(D\psi)-A(\xi_{0})|)\,dx
\end{align*}
holds for every $z_{0},\xi_{0} \in \mathbb{R}^{n}$, whenever $2Q \Subset \Omega$.
\end{lemma}
\begin{proof}
By following the proof of \cite[Lemma~3.3]{BSY}, with considering cubes instead of balls, we have
\begin{align*}
\lefteqn{\mean{Q}|V(Dw_{1})-V(z_{0})|^{2}\,dx} \nonumber \\
& \le c\left(\mean{2Q}|V(Dw_{1})-V(z_{0})|^{2\sigma}\,dx\right)^{\frac{1}{\sigma}} + c\mean{2Q}(\varphi_{|z_{0}|})^{*}(|A(D\psi)-A(\xi_{0})|)\,dx \\
\overset{\eqref{mono.shift}}&{\le} c\left(\mean{2Q}[(\varphi_{|z_{0}|})^{*}(|A(Dw_{1})-A(z_{0})|)]^{\sigma}\,dx\right)^{\frac{1}{\sigma}} + c\mean{2Q}(\varphi_{|z_{0}|})^{*}(|A(D\psi)-A(\xi_{0})|)\,dx,
\end{align*}
for any $\sigma \in (0,1)$, where $c\equiv c(\data,\sigma)$. We then observe that
\begin{equation}\label{exp.down}
t \mapsto [((\varphi_{|z_{0}|})^{*})^{-1}(t^{1/\sigma})]^{\kappa} \text{ is convex for } \sigma>0 \text{ small enough.} 
\end{equation}
Hence, we apply Young's inequality to the first integral on the right-hand side, thereby getting the desired estimate.
\end{proof}

We next examine some various regularity estimates for the homogeneous equation 
\begin{equation}\label{limiting.equation}
-\mathrm{div}\,A(Dv)=0 \quad \textrm{in }\Omega.
\end{equation}
The following reverse H\"older's inequality can be found in \cite[Lemma~3.2]{Min07}.
\begin{lemma}
Let $v\in W^{1,p}_{\mathrm{loc}}(\Omega)$ be a weak solution to \eqref{limiting.equation} under assumptions \eqref{growth} with $p>1$. Then for any $\sigma \in (0,1)$ there exists a constant $c \equiv c(\data,\sigma)$ such that 
\begin{equation}\label{revhol.v}
\mean{Q}|V(Dv)-V(z_{0})|^{2}\,dx \le c\left(\mean{2Q}|V(Dv)-V(z_{0})|^{2\sigma}\,dx\right)^{\frac{1}{\sigma}}
\end{equation}
holds for every $z_{0} \in \mathbb{R}^{n}$, whenever $2Q \Subset \Omega$. 
\end{lemma}
We then recall a gradient H\"older regularity result for \eqref{limiting.equation}. We state it as in \cite[Theorem~3.3]{AKM18} with a slight modification. \begin{lemma}\label{grad.hol.v}
Let $v \in W^{1,p}_{\mathrm{loc}}(\Omega)$ be a weak solution to \eqref{limiting.equation}
under assumptions \eqref{growth} with $p>1$. Then $v \in C^{1,\alpha}_{\mathrm{loc}}(\Omega)$ for some $\alpha \equiv \alpha(\data) \in (0,1)$. Moreover, for every $t>0$, there exists a constant $c \equiv c(\data,t)$ such that
\begin{equation*}
\sup_{\varepsilon Q}(|Dv|+s) \le \frac{c}{(1-\varepsilon)^{n/t}}\left(\mean{Q}(|Dv|+s)^{t}\,dx\right)^{\frac{1}{t}}
\end{equation*} 
holds for every cube $Q \Subset \Omega$ and $\varepsilon \in (0,1)$. 
Finally, there exists a constant $c \equiv c(\data)$ such that 
\begin{equation*}
|Dv(x_{1})-Dv(x_{2})| \le c \varepsilon^{\alpha}\mean{Q}|Dv-(Dv)_{Q}|\,dx
\end{equation*}
holds for every cube $Q\Subset \Omega$ and $x_{1},x_{2} \in \varepsilon Q$ with $\varepsilon \in (0,1/2]$.
\end{lemma}

We recall \eqref{mod.exs} to further establish a decay estimate for a modified excess functional of $A(Dv)$.
\begin{lemma}\label{ed.ADv}
Let $v\in W^{1,p}_{\mathrm{loc}}(\Omega)$ be a weak solution to \eqref{limiting.equation} under assumptions \eqref{growth} with $p>1$. Then, with $\kappa$ given in \eqref{def.kappa}, there exists an exponent $\alpha_{A} \equiv \alpha_{A}(\data) \in (0,1)$ such that 
\begin{equation*}
\left(\mean{Q_{\rho}}|A(Dv)-\mathcal{P}_{\kappa,Q_{\rho}}(A(Dv))|^{\kappa}\,dx \right)^{\frac{1}{\kappa}} \le c\left(\frac{\rho}{R}\right)^{\alpha_{A}}\left(\mean{Q_{R}}|A(Dv)-\mathcal{P}_{\kappa,Q_{R}}(A(Dv))|^{\kappa}\,dx\right)^{\frac{1}{\kappa}}
\end{equation*}
holds whenever $Q_{\rho} \subset Q_{R} \Subset \Omega$ are concentric cubes, where $c\equiv c(\data)$.
\end{lemma}
\begin{proof}
We may assume $\rho \le R/2$ without loss of generality, and recall the following $L^{1}$-excess decay estimate that follows from \cite[Theorem~4.4]{BSY}:
\begin{equation*}
\mean{Q_{\rho}}|A(Dv)-(A(Dv))_{Q_{\rho}}|\,dx \le c\left(\frac{\rho}{R}\right)^{\alpha_{A}}\mean{Q_{R/2}}|A(Dv)-(A(Dv))_{Q_{R/2}}|\,dx.
\end{equation*}
Using this, we have
\begin{align*}
\lefteqn{ \left(\mean{Q_{\rho}}|A(Dv)-\mathcal{P}_{\kappa,Q_{\rho}}(A(Dv))|^{\kappa}\,dx\right)^{\frac{1}{\kappa}} \le \left(\mean{Q_{\rho}}|A(Dv)-(A(Dv))_{Q_{\rho}}|^{\kappa}\,dx\right)^{\frac{1}{\kappa}} }\\
& \le \mean{Q_{\rho}}|A(Dv)-(A(Dv))_{Q_{\rho}}|\,dx \\
& \le c\left(\frac{\rho}{R}\right)^{\alpha_A}\mean{Q_{R/2}}|A(Dv)-(A(Dv))_{Q_{R/2}}|\,dx \\
\overset{\eqref{mean.min}}&{\le} c\left(\frac{\rho}{R}\right)^{\alpha_{A}}\mean{Q_{R/2}}|A(Dv)-A(z_{0})|\,dx \\
& \le c\left(\frac{\rho}{R}\right)^{\alpha_{A}}((\varphi_{|z_{0}|})^{*})^{-1}\left(\mean{Q_{R/2}}(\varphi_{|z_{0}|})^{*}(|A(Dv)-A(z_{0})|)\,dx\right) \\
\overset{\eqref{mono.shift},\eqref{revhol.v}}&{\le} c\left(\frac{\rho}{R}\right)^{\alpha_{A}}((\varphi_{|z_{0}|})^{*})^{-1}\left[\left(\mean{Q_{R}}[(\varphi_{|z_{0}|})^{*}(|A(Dv)-A(z_{0})|)]^{\sigma}\,dx\right)^{\frac{1}{\sigma}}\right] \end{align*}
whenever $z_{0} \in \mathbb{R}^{n}$ and $\sigma \in (0,1)$, where $c\equiv c(\data,\sigma)$. Then, recalling \eqref{exp.down}, the desired estimate follows by applying Jensen's inequality and then taking infimum with respect to $z_{0}$.
\end{proof}

\section{Basic comparison estimates} \label{sec.comparison.est}
In this section, we derive several comparison estimates under an additional assumption
\begin{equation}
\label{regular}
\mu \in W^{-1,p'}(\Omega) \cap L^{1}(\Omega), \qquad u \in \mathcal{A}^{g}_{\psi}(\Omega).
\end{equation}
This assumption will be eventually removed in \secref{sec.pf.thm1} below.

Here we introduce the mixed norm
\[ \|f\|_{L^{s_{2}}_{x'}L^{s_{1}}_{x_{n}}(Q_{\rho}(x_{0}))} \coloneqq \left(\int_{Q_{\rho}'(x_{0}')}\left(\int_{x_{0,n}+(-\rho,\rho)}|f(x',x_{n})|^{s_{1}}\,dx_{n}\right)^\frac{s_{2}}{s_{1}}\,dx'\right)^{\frac{1}{s_{2}}} \]
and its averaged version
\[ \avenorm{f}_{L^{s_2}_{x'}L^{s_1}_{x_n}(Q_\rho(x_0))}=\left(\mean{Q_{\rho}'(x_{0}')}\left(\mean{x_{0,n}+(-\rho,\rho)} |f(x',x_n)|^{s_1}\,dx_n\right)^{\frac{s_2}{s_1}} dx'\right)^{\frac{1}{s_2}}. \]

In \cite{BSY}, the starting point of various comparison estimates and further linearization was the weighted type energy estimate given in \cite[Lemma~5.1]{BSY}. It is valid for \eqref{p.bound} as well, but the proof of subsequent comparison estimates in \cite[Section~5]{BSY} do not work in the case \eqref{p.bound}.

We therefore develop a slightly different approach motivated from those in \cite{BSY,NP23ARMA,PS22}, at some stage dividing the cases
\begin{equation} \label{divide.case}
\frac{3n-2}{2n-1} < p \le 2-\frac{1}{n} \qquad \text{and} \qquad 1< p \le \frac{3n-2}{2n-1}. 
\end{equation}

\subsection{Some technical results}
The following lemma is analogous to \cite[Lemma~2.1]{NP23ARMA}, see also the proof of \cite[Theorem~4.1]{KM18JEMS}. Note that the estimate in \cite[Lemma~2.1]{NP23ARMA} is concerned with the case $k=0$ only, as $u-k$ also solves equation \eqref{model} for other values of $k$. Since this is not the case for obstacle problems, we have to consider general $k$ in the estimate. Also, due to the obstacle constraint, we need different choices of test functions.

\begin{lemma}\label{caccio.u}
Let $u \in \mathcal{A}^{g}_{\psi}(\Omega)$ be the weak solution to \eqref{opmu} under assumptions \eqref{growth} with $p>1$. Then for any $\varepsilon>0$, $k\in\mathbb{R}$ and any nonnegative $\eta \in C^{\infty}_{0}(\Omega)$, we have
\begin{align*}
\lefteqn{ \int_{\Omega}\left|D\left[(1+|u-k|)^{\frac{p-1-\varepsilon}{p}}\eta\right]\right|^{p}\,dx } \\
& \le \frac{c}{\varepsilon^{p}}\int_{\Omega}(1+|u-k|)^{(\varepsilon+1)(p-1)}|D\eta|^{p}\,dx + \frac{c}{\varepsilon}\int_{\Omega}\eta^{p}d|\mu| + c\int_{\Omega}s^{p}\eta^{p}\,dx
\end{align*}
for a constant $c\equiv c(\data)$.
\end{lemma}
\begin{proof}
We first test \eqref{opmu} with
\begin{align*} 
\phi & = u + \frac{1}{\varepsilon}[1-(1+(u-k)_{-})^{-\varepsilon}]\eta^{p} \ge u \ge \psi, 
\end{align*}
to have
\begin{align*}
\int_{\{u \le k\}}\frac{-A(Du)\cdot Du}{(1+|u-k|)^{\varepsilon+1}}\eta^{p}\,dx 
& \ge -\frac{p}{\varepsilon}\int_{\Omega}A(Du)\cdot[1-(1+(u-k)_{-})^{-\varepsilon}]\eta^{p-1}D\eta\,dx \\
& \quad + \frac{1}{\varepsilon}\int_{\Omega}[1-(1+(u-k)_{-})^{-\varepsilon}]\eta^{p}\,d\mu
\end{align*}
and so
\begin{align}\label{test.minus}
\lefteqn{ \int_{\{u \le k\}}\frac{(|Du|+s)^{p}\eta^{p}}{(1+|u-k|)^{\varepsilon+1}}\,dx } \nonumber \\
& \le \frac{c}{\varepsilon}\int_{\Omega}(|Du|+s)^{p-1}\eta^{p-1}|D\eta|\,dx + \frac{c}{\varepsilon}\int_{\Omega}\eta^{p}\,d|\mu| + c\int_{\Omega}s^{p}\eta^{p}\,dx.
\end{align}
We next test \eqref{opmu} with
\begin{align*}
\phi & = u + \frac{1}{\varepsilon}(1+(u-k)_{+})^{-\varepsilon}\eta^{p} \ge u \ge \psi, 
\end{align*}
and estimate in a similar way to obtain
\begin{align}\label{test.plus}
\lefteqn{ \int_{\{u \ge k\}}\frac{(|Du|+s)^{p}\eta^{p}}{(1+|u-k|)^{\varepsilon+1}}\,dx } \nonumber \\ 
& \le \frac{c}{\varepsilon}\int_{\Omega}(|Du|+s)^{p-1}\eta^{p-1}|D\eta|\,dx + \frac{c}{\varepsilon}\int_{\Omega}\eta^{p}\,d|\mu| + c\int_{\Omega}s^{p}\eta^{p}\,dx.
\end{align}
Combining \eqref{test.minus} and\eqref{test.plus}, we arrive at
\begin{align*}
\int_{\Omega}\frac{(|Du|+s)^{p}\eta^{p}}{(1+|u-k|)^{\varepsilon+1}}\,dx \le \frac{c}{\varepsilon}\int_{\Omega}(|Du|+s)^{p-1}\eta^{p-1}|D\eta|\,dx + \frac{c}{\varepsilon}\int_{\Omega}\eta^{p}\,d|\mu| + c\int_{\Omega}s^{p}\eta^{p}\,dx.
\end{align*}
Applying Young's inequality to the first term on the right-hand side, and then recalling the identity
\begin{align*} 
\lefteqn{ D\left((1+|u-k|)^{\frac{p-1-\varepsilon}{p}}\eta\right) }\\
& = \eta\frac{p-1-\varepsilon}{p}(1+|u-k|)^{-\frac{1+\varepsilon}{p}}\text{sign}(u-k)Du + (1+|u-k|)^{\frac{p-1-\varepsilon}{p}}D\eta,
\end{align*}
we have the desired estimate.
\end{proof}

\lemref{caccio.u} gives a reverse H\"older type estimate for $u$ and a mixed norm estimate for $Du$; their proofs are exactly the same as in \cite[Section~2]{NP23ARMA}. They will play a crucial role in \lemref{Du-Dw1} below.

\begin{lemma}\label{revhol.u}
Let $u \in \mathcal{A}^{g}_{\psi}(\Omega)$ be the weak solution to \eqref{opmu} under assumptions \eqref{growth} with $1<p<n$. Then for any
\[ 0 < q_{1} < q < \frac{n(p-1)}{n-p}, \]
$k \in \mathbb{R}$ and $\sigma \in (0,1)$, we have
\[ \left(\mean{Q_{\sigma r}}(|u-k|+rs)^{q}\,dx\right)^{\frac{1}{q}} \le c\left(\mean{Q_{r}}(|u-k|+rs)^{q_{1}}\,dx\right)^{\frac{1}{q_{1}}} + c\left[\frac{|\mu|(Q_{r})}{r^{n-p}}\right]^{\frac{1}{p-1}} \]
for a constant $c\equiv c(\data,q,q_{1},\sigma)$, whenever $Q_{\sigma r} \subset Q_{r} \subset \Omega$ are concentric cubes.
\end{lemma}

\begin{lemma}\label{ccp.mixed}
Let $u \in \mathcal{A}^{g}_{\psi}(\Omega)$ be the weak solution to \eqref{opmu} under assumptions \eqref{growth} with $1<p<n$. Then for any exponents $q_{1}, s_{1}, s_{2}$ satisfying
\[0<q_{1} < \frac{n(p-1)}{n-p}, \quad \frac{p-1}{n-1} < s_{1} < p, \quad  0 < s_{2} < \frac{s_{1}(n-1)(p-1)}{s_{1}(n-1)-p+1}, \]
and any $k \in \mathbb{R}$, we have
\begin{equation*}
\avenorm{|Du|+s}_{L^{s_{2}}_{x'}L^{s_{1}}_{x_{n}}(Q_{\sigma r})} \le c\left[\frac{|\mu|(Q_{r})}{r^{n-1}}\right]^{\frac{1}{p-1}} + \frac{c}{r}\left(\mean{Q_{r}}(|u-k|+rs)^{q_{1}}\,dx\right)^{\frac{1}{q_{1}}}
\end{equation*}
for a constant $c \equiv c(\data,q_{1},s_{1},s_{2},\sigma)$, whenever $Q_{\sigma r} \subset Q_{r} \subset \Omega$ are concentric cubes.
\end{lemma}

For a fixed cube $Q_{4R} \subset \Omega$, we first consider the homogeneous obstacle problem
\begin{equation}\label{homogeneous}
\left\{
\begin{aligned}
\int_{Q_{4R}} A(Dw_{1})\cdot D(\phi - w_{1})\,dx & \ge 0 \quad \forall \; \phi \in \mathcal{A}^{u}_{\psi}(Q_{4R}), \\
w_{1} & \ge \psi \quad \textrm{a.e. in } Q_{4R},\\
w_{1} & = u \quad \textrm{on } \partial Q_{4R}.
\end{aligned}
\right.
\end{equation}
We obtain a preliminary comparison estimate between \eqref{opmu} and \eqref{homogeneous}.

\subsection{Comparison with \eqref{homogeneous} in the case $\eqref{divide.case}_{1}$}
In this case, we extend the approaches in \cite{BSY,PS22}. We first obtain the following lemma, which generalizes \cite[Lemma~5.1]{BSY}.
\begin{lemma}
Let $u\in \mathcal{A}^{g}_{\psi}(\Omega) $ be the weak solution to \eqref{opmu} under assumptions \eqref{growth} with $p>1$, and let $w_{1} \in \mathcal{A}^{u}_{\psi}(Q_{4R})$ be as in \eqref{homogeneous}. Then 
\begin{equation}\label{weighted.comparison}
\int_{Q_{4R}}\frac{|u-w_{1}|^{-\gamma}|V(Du)-V(Dw_{1})|^{2}}{(h^{1-\gamma}+|u-w_{1}|^{1-\gamma})^{\xi}}\,dx \le c\frac{h^{(1-\gamma)(1-\xi)}}{(1-\gamma)(\xi-1)}|\mu|(Q_{4R})
\end{equation}
holds for a constant $c\equiv c(\data)$, whenever $h>0$, $\xi >1$ and $\gamma \in [0,1)$.
\end{lemma}
\begin{proof}
For any positive constants $\varepsilon$ and $\tilde{\varepsilon}$ satisfying $\varepsilon > \tilde{\varepsilon}^{1-\gamma}$, consider the function 
\[ \zeta_{\pm} \coloneqq \min\left\{1,\max\left\{\frac{(u-w_{1})_{\pm}^{1-\gamma}-\tilde{\varepsilon}^{1-\gamma}}{\varepsilon-\tilde{\varepsilon}^{1-\gamma}},0\right\}\right\}. \]
We immediately see that $\supp\, \zeta_{\pm} = Q_{4R} \cap \{(u-w_{1})_{\pm} \ge \tilde{\varepsilon} \}$ and 
\[ D\zeta_{\pm} = \frac{1-\gamma}{\varepsilon-\tilde{\varepsilon}^{1-\gamma}}\chi_{\mathcal{A}_{\pm}(\tilde{\varepsilon},\varepsilon)}(u-w_{1})_{\pm}^{-\gamma}D(u-w_{1})_{\pm}, \]
\[ \text{where}\;\; \mathcal{A}_{\pm}(\tilde{\varepsilon},\varepsilon) \coloneqq Q_{4R} \cap \left\{\tilde{\varepsilon} < (u-w_{1})_{\pm} < \varepsilon^{\frac{1}{1-\gamma}} \right\}.\]
We also consider the function
\[ \eta_{\pm} \coloneqq \frac{1}{\xi-1}\left[1-\left(1+\frac{(u-w_{1})_{\pm}^{1-\gamma}}{h^{1-\gamma}}\right)^{1-\xi}\right]. \]
The mean value theorem, applied to the function $t \mapsto t^{1-\xi}/(1-\xi)$, gives
\[ \eta_{\pm}(x) = \left(\frac{(u-w_{1})_{\pm}(x)}{h}\right)^{1-\gamma}(\tilde{\eta}_{\pm}(x))^{-\xi} \quad \text{for some } 1< \tilde{\eta}_{\pm}(x) < 1+\left(\frac{(u-w_{1})(x)}{h}\right)^{1-\gamma}. \]
Then, since
\begin{equation}\label{admissible}
\tilde{\varepsilon}^{\gamma}(u-w_{1})_{\pm}^{1-\gamma} \le (u-w_{1})_{\pm} \;\; \text{in} \;\; \supp\,\zeta_{\pm},
\end{equation} 
we observe that
\begin{align*} 
u- \tilde{\varepsilon}^{\gamma}h^{1-\gamma}\eta_{+}\zeta_{+} & = u - \tilde{\varepsilon}^{\gamma}(u-w_{1})_{+}^{1-\gamma}\tilde{\eta}_{+}^{-\xi}\zeta_{+} \ge u - (u-w_{1})_{+} = \min\{u,w_{1}\}, \\
w_{1} - \tilde{\varepsilon}^{\gamma}h^{1-\gamma}\eta_{-}\zeta_{-} & = w_{1} - \tilde{\varepsilon}^{\gamma}(u-w_{1})_{-}^{1-\gamma}\tilde{\eta}_{-}^{-\xi}\zeta_{-} \ge w_{1} - (u-w_{1})_{-} = \min\{u,w_{1}\}
\end{align*}
a.e. in $Q_{4R}$. From this and \eqref{admissible}, we see that the functions
\[ u \pm \tilde{\varepsilon}^{\gamma}h^{1-\gamma}\eta_{\mp}\zeta_{\mp} \quad \text{and} \quad w_{1} \pm \tilde{\varepsilon}^{\gamma}h^{1-\gamma}\eta_{\pm}\zeta_{\pm} \]
belong to the admissible set $\mathcal{A}^{u}_{\psi}(Q_{4R})$.

We now test \eqref{opmu} with $\phi \equiv u \pm \tilde{\varepsilon}^{\gamma}h^{1-\gamma}\eta_{\mp}\zeta_{\mp}$ to get
\begin{align*}
\lefteqn{ \int_{\mathcal{A}_{\pm}(\tilde{\varepsilon},\varepsilon)}\frac{|u-w_{1}|^{-\gamma}A(Du)\cdot (Du-Dw_{1})}{(h^{1-\gamma}+|u-w_{1}|^{1-\gamma})^{\xi}}\zeta_{\pm}\,dx }\\
& \quad + \int_{\mathcal{A}_{\pm}(\tilde{\varepsilon},\varepsilon)}\frac{h^{(1-\gamma)(1-\xi)}}{\varepsilon-\tilde{\varepsilon}^{1-\gamma}}\eta_{\pm}|u-w_{1}|^{-\gamma}A(Du)\cdot(Du-Dw_{1})\,dx \le \frac{h^{(1-\gamma)(1-\xi)}}{(1-\gamma)(\xi-1)}|\mu|(Q_{4R}) .
\end{align*}
In a similar way, testing \eqref{homogeneous} with $\phi \equiv w_{1} \pm \tilde{\varepsilon}^{\gamma}h^{1-\gamma}h^{1-\gamma}\eta_{\pm}\zeta_{\pm}$, we have
\begin{align*}
\lefteqn{ -\int_{\mathcal{A}_{\pm}(\tilde{\varepsilon},\varepsilon)}\frac{|u-w_{1}|^{-\gamma}A(Dw_{1})\cdot (Du-Dw_{1})}{(h^{1-\gamma}+|u-w_{1}|^{1-\gamma})^{\xi}}\zeta_{\pm}\,dx }\\
& \quad - \int_{\mathcal{A}_{\pm}(\tilde{\varepsilon},\varepsilon)}\frac{h^{(1-\gamma)(1-\xi)}}{\varepsilon-\tilde{\varepsilon}^{1-\gamma}}\eta_{\pm}|u-w_{1}|^{-\gamma}A(Dw_{1})\cdot(Du-Dw_{1})\,dx \le 0.
\end{align*}
Combining the above two displays and using \eqref{mono.V}, we see that
\begin{align*}
\lefteqn{ \int_{\mathcal{A}_{\pm}(\tilde{\varepsilon},\varepsilon)}\frac{|u-w_{1}|^{-\gamma}|V(Du)-V(Dw_{1})|^{2}}{(h^{1-\gamma}+|u-w_{1}|^{1-\gamma})^{\xi}}\zeta_{\pm}\,dx }\\
& \quad + \int_{\mathcal{A}_{\pm}(\tilde{\varepsilon},\varepsilon)}\frac{h^{(1-\gamma)(1-\xi)}}{\varepsilon-\tilde{\varepsilon}^{1-\gamma}}\eta_{\pm}|u-w_{1}|^{-\gamma}|V(Du)-V(Dw_{1})|^{2}\,dx \le c\frac{h^{(1-\gamma)(1-\xi)}}{(1-\gamma)(\xi-1)}|\mu|(Q_{4R})
\end{align*}
holds for a constant $c\equiv c(\data)$. In particular, since the second term on the left-hand side is nonnegative, we have
\[ \int_{\mathcal{A}_{\pm}(\tilde{\varepsilon},\varepsilon)}\frac{|u-w_{1}|^{-\gamma}|V(Du)-V(Dw_{1})|^{2}}{(h^{1-\gamma}+|u-w_{1}|^{1-\gamma})^{\xi}}\zeta_{\pm}\,dx \le c\frac{h^{(1-\gamma)(1-\xi)}}{(1-\gamma)(\xi-1)}|\mu|(Q_{4R}). \]
As $\tilde{\varepsilon} \rightarrow 0$, recalling the definition of $\zeta_{\pm}$, we arrive at
\[ \int_{Q_{4R}}\frac{|u-w_{1}|^{-\gamma}|V(Du)-V(Dw_{1})|^{2}}{(h^{1-\gamma}+|u-w_{1}|^{1-\gamma})^{\xi}}\min\left\{1,\frac{|u-w_{1}|^{1-\gamma}}{\varepsilon}\right\}\,dx \le c\frac{h^{(1-\gamma)(1-\xi)}}{(1-\gamma)(\xi-1)}|\mu|(Q_{4R}) \]
with $c\equiv c(\data)$. Thus, letting $\varepsilon\rightarrow0$ in the last display gives \eqref{weighted.comparison}.
\end{proof}

\begin{lemma}
Let $u \in \mathcal{A}^{g}_{\psi}(\Omega)$ be the weak solution to \eqref{opmu} under assumptions \eqref{growth} with $\eqref{divide.case}_{1}$, and let let $w_{1} \in \mathcal{A}^{u}_{\psi}(Q_{4R})$ be as in \eqref{homogeneous}. Then for any 
\begin{equation}\label{qrange}
q \in \left(\frac{n}{2n-1}, \frac{n(p-1)}{n-1} \right), 
\end{equation}
the estimate
\begin{align}\label{1st.comparison}
\lefteqn{ \left(\mean{Q_{4R}}|Du-Dw_{1}|^{q}\,dx\right)^{\frac{1}{q}} + \frac{1}{R}\left(\mean{Q_{4R}}|u-w_{1}|^{q}\,dx\right)^{\frac{1}{q}} } \nonumber \\
& \le c\left[\frac{|\mu|(Q_{4R})}{(4R)^{n-1}}\right]^{\frac{1}{p-1}} + c\left[\frac{|\mu|(Q_{4R})}{(4R)^{n-1}}\right]\left(\mean{Q_{4R}}(|Du|+s)^{q}\,dx\right)^{\frac{2-p}{q}}
\end{align} 
holds for a constant $c\equiv c(\data,q)$.
\end{lemma}
\begin{proof}
Given a constant $\varepsilon>0$, define $\mathcal{B}_{\varepsilon} \coloneqq Q_{4R} \cap \{|u-w_{1}| > \varepsilon \}$. We set the exponent 
\begin{equation}\label{def.beta}
\beta \coloneqq \frac{np(1-q)}{n-q} \;\; \Longleftrightarrow \;\; \frac{\beta q}{(1-q)(p-\beta)} = \frac{n}{n-1}
\end{equation}
and define
\[ M_{\varepsilon} \coloneqq \frac{p}{p-\beta}\mean{Q_{4R}}\left|D\left[(u-w_{1})^{\frac{p-\beta}{p}}\right]\right|\chi_{\mathcal{B}_{\varepsilon}}\,dx. \]
Note that $M_{\varepsilon} < \infty$ since $|u-w_{1}| > \varepsilon$ in $\mathcal{B}_{\varepsilon}$. 
We start by estimating
\begin{align*}
\mean{Q_{4R}}|Du-Dw_{1}|^{q}\chi_{\mathcal{B}_{\varepsilon}}\,dx & = \mean{Q_{4R}}\left(|u-w_{1}|^{-\frac{\beta}{p}}|Du-Dw_{1}|\right)^{q}|u-w_{1}|^{\frac{\beta q}{p}}\chi_{\mathcal{B}_{\varepsilon}}\,dx \\
& \le M_{\varepsilon}^{q}\left(\mean{Q_{4R}}|u-w_{1}|^{\frac{\beta q}{(1-q)p}}\chi_{\mathcal{B}_{\varepsilon}}\,dx\right)^{1-q}
\end{align*}
Here, recalling \eqref{def.beta}, we apply Sobolev-Poincar\'e inequality to have
\begin{align*}
\mean{Q_{4R}}|u-w_{1}|^{\frac{\beta q}{(1-q)p}}\chi_{\mathcal{B}_{\varepsilon}}\,dx 
& \le c\mean{Q_{4R}}\left(|u-w_{1}|^{\frac{p-\beta}{p}}-\varepsilon^{\frac{p-\beta}{p}}\right)_{+}^{\frac{\beta q}{(1-q)(p-\beta)}} + c\varepsilon^{\frac{\beta q}{(1-q)p}} \\
& \le c(RM_{\varepsilon})^{\frac{\beta q}{(1-q)(p-\beta)}} + c\varepsilon^{\frac{\beta q}{(1-q)p}}.
\end{align*}
Then, letting
\begin{equation}\label{def.he}
h_{\varepsilon} \coloneqq (RM_{\varepsilon})^{\frac{p}{p-\beta}} + \varepsilon,
\end{equation}
we arrive at
\begin{equation}\label{u-w.M}
\mean{Q_{4R}}|u-w_{1}|^{\frac{\beta q}{(1-q)p}}\chi_{\mathcal{B}_{\varepsilon}}\,dx \le ch_{\varepsilon}^{\frac{\beta q}{(1-q)p}}
\end{equation}
and
\begin{equation}\label{Du-Dw.M}
\mean{Q_{4R}}|Du-Dw_{1}|^{q}\chi_{\mathcal{B}_{\varepsilon}}\,dx \le cM_{\varepsilon}^{q}h_{\varepsilon}^{\frac{\beta q}{p}}
\end{equation}
for some $c \equiv c(\data,q).$

We now estimate $M_{\varepsilon}$. Recalling the inequality (see for instance \cite[(9.39)]{Min11Milan})
\[ |Du-Dw_{1}| \le c|V(Du)-V(Dw_{1})|^{\frac{2}{p}} + c(|Du|+s)^{\frac{2-p}{2}}|V(Du)-V(Dw_{1})|, \]
we directly have
\begin{align}\label{est.I1.I2}
M_{\varepsilon} & \le c\mean{Q_{4R}}|u-w_{1}|^{-\frac{\beta}{p}}|V(Du)-V(Dw_{1})|^{\frac{2}{p}}\chi_{\mathcal{B}_{\varepsilon}}\,dx \nonumber \\
& \quad + c\mean{Q_{4R}}|u-w_{1}|^{-\frac{\beta}{p}}(|Du|+s)^{\frac{2-p}{2}}|V(Du)-V(Dw_{1})|\chi_{\mathcal{B}_{\varepsilon}}\,dx \nonumber \\
& \eqqcolon cI_{1} + cI_{2}.
\end{align}
Then, with $\xi_{1} > 1$ to be chosen, we use H\"older's inequality and \eqref{weighted.comparison} to obtain 
\begin{align}\label{I1.start}
I_{1} & = \mean{Q_{4R}}\left(\frac{|u-w_{1}|^{-\beta}|V(Du)-V(Dw_{1})|^{2}}{(h_{\varepsilon}^{1-\beta}+|u-w_{1}|^{1-\beta})^{\xi_{1}}}\right)^{\frac{1}{p}}(h_{\varepsilon}^{1-\beta}+|u-w_{1}|^{1-\beta})^{\frac{\xi_{1}}{p}}\chi_{\mathcal{B}_{\varepsilon}}\,dx \nonumber \\
& \le \left(\mean{Q_{4R}}\frac{|u-w_{1}|^{-\beta}|V(Du)-V(Dw_{1})|^{2}}{(h_{\varepsilon}^{1-\beta}+|u-w_{1}|^{1-\beta})^{\xi_{1}}}\chi_{\mathcal{B}_{\varepsilon}}\,dx\right)^{\frac{1}{p}} \nonumber \\
& \qquad \cdot \left(\mean{Q_{4R}}(h_{\varepsilon}^{1-\beta}+|u-w_{1}|^{1-\beta})^{\frac{\xi_{1}}{p-1}}\chi_{\mathcal{B}_{\varepsilon}}\,dx\right)^{\frac{p-1}{p}} \nonumber \\
& \le ch_{\varepsilon}^{\frac{(1-\beta)(1-\xi_{1})}{p}}\left[\frac{|\mu|(Q_{4R})}{(4R)^{n}}\right]^{\frac{1}{p}}\left\{ h_{\varepsilon}^{\frac{(1-\beta)\xi_{1}}{p}} + \left(\mean{Q_{4R}}|u-w_{1}|^{\frac{(1-\beta)\xi_{1}}{p-1}}\chi_{\mathcal{B}_{\varepsilon}}\,dx\right)^{\frac{p-1}{p}} \right\}.
\end{align}
Since 
\[ q < \frac{n(p-1)}{n-1} \;\; \Longleftrightarrow \;\; \frac{1-\beta}{p-1} < \frac{\beta q}{(1-q)p}, \]
we can choose $\xi_{1} >1$, depending only on $\data$ and $q$, such that 
\[ \frac{(1-\beta)\xi_{1}}{p-1} < \frac{\beta q}{(1-q)p}. \]  
Then, applying H\"older's inequality, we obtain
\begin{align*}
\left(\mean{Q_{4R}}|u-w_{1}|^{\frac{(1-\beta)\xi_{1}}{p-1}}\chi_{\mathcal{B}_{\varepsilon}}\,dx\right)^{\frac{p-1}{p}} & \le \left(\mean{Q_{4R}}|u-w_{1}|^{\frac{\beta q}{(1-q)p}}\chi_{\mathcal{B}_{\varepsilon}}\,dx\right)^{\frac{(1-q)(1-\beta)\xi_{1}}{\beta q}} \nonumber \\
\overset{\eqref{u-w.M}}&{\le} ch_{\varepsilon}^{\frac{(1-\beta)\xi_{1}}{p}}
\end{align*}
for some $c \equiv c(\data,q)$. Plugging this into \eqref{I1.start} gives the following estimate of $I_{1}$:
\begin{equation}\label{I1.est}
I_{1} \le ch_{\varepsilon}^{\frac{1-\beta}{p}}\left[\frac{|\mu|(Q_{4R})}{(4R)^{n}}\right]^{\frac{1}{p}}.
\end{equation}
On the other hand, with $\gamma = 2\beta/p \in (0,1)$ and $\xi_{2} >1$ to be chosen, a similar calculation as in \eqref{I1.start} gives
\begin{align}\label{I2.start}
I_{2} & = \mean{Q_{4R}}\left(\frac{|u-w_{1}|^{-\gamma}|V(Du)-V(Dw_{1})|^{2}}{(h_{\varepsilon}^{1-\gamma}+|u-w_{1}|^{1-\gamma})^{\xi_{2}}}\,dx\right)^{\frac{1}{2}} \nonumber \\
& \qquad \qquad \cdot (h_{\varepsilon}^{1-\gamma}+|u-w_{1}|^{1-\gamma})^{\frac{\xi_{2}}{2}}(|Du|+s)^{\frac{2-p}{2}}\chi_{\mathcal{B}_{\varepsilon}}\,dx \nonumber \\
& \le ch_{\varepsilon}^{\frac{(1-\gamma)(1-\xi_{2})}{2}}\left[\frac{|\mu|(Q_{4R})}{(4R)^{n}}\right]^{\frac{1}{2}}\left(\mean{Q_{4R}}(h_{\varepsilon}^{1-\gamma}+|u-w_{1}|^{1-\gamma})^{\xi_{2}}(|Du|+s)^{2-p}\chi_{\mathcal{B}_{\varepsilon}}\,dx\right)^{\frac{1}{2}}.
\end{align}
We then apply H\"older's inequality to the integral appearing on the right-hand side as follows:
\begin{align}\label{I2.rhs}
\lefteqn{ \mean{Q_{4R}}(h_{\varepsilon}^{1-\gamma}+|u-w_{1}|^{1-\gamma})^{\xi_{2}}(|Du|+s)^{2-p}\chi_{\mathcal{B}_{\varepsilon}}\,dx } \nonumber \\
& \le \left(\mean{Q_{4R}}(h_{\varepsilon}^{1-\gamma}+|u-w_{1}|^{1-\gamma})^{\frac{\xi_{2} q}{q-2+p}}\chi_{\mathcal{B}_{\varepsilon}}\,dx\right)^{\frac{q-2+p}{q}}\left(\mean{Q_{4R}}(|Du|+s)^{q}\,dx\right)^{\frac{2-p}{q}}
\end{align}
and observe that
\[ q < \frac{n(p-1)}{n-1} \;\; \Longleftrightarrow \;\; \frac{(1-\gamma)q}{q-2+p} = \frac{(2n-1)q-n}{q-2+p}\frac{q}{n-q} < \frac{nq}{n-q}. \]
Thus, we can choose the constant $\xi_{2} > 1$, depending only on $\data$ and $q$, such that
\[ \frac{(1-\gamma)\xi_{2} q}{q-2+p} < \frac{nq}{n-q}. \]
We note that \eqref{def.beta} implies $\beta q/[(1-q)p] = nq/(n-q)$. 
Then H\"older's inequality and \eqref{u-w.M} imply
\begin{align}\label{I2.rhs2}
\lefteqn{ \left(\mean{Q_{4R}}(h_{\varepsilon}^{1-\gamma}+|u-w_{1}|^{1-\gamma})^{\frac{\xi_{2} q}{q-2+p}}\chi_{\mathcal{B}_{\varepsilon}}\,dx\right)^{\frac{q-2+p}{q}} } \nonumber \\
& \le ch_{\varepsilon}^{(1-\gamma)\xi_{2}} + c\left(\mean{Q_{4R}}|u-w_{1}|^{\frac{nq}{n-q}}\chi_{\mathcal{B}_{\varepsilon}}\,dx\right)^{\frac{(1-\gamma)\xi_{2}(n-q)}{nq}} \le ch_{\varepsilon}^{(1-\gamma)\xi_{2}}.
\end{align}
Connecting \eqref{I2.rhs} and \eqref{I2.rhs2} to \eqref{I2.start}, $I_{2}$ is estimated as
\begin{equation}\label{I2.est}
I_{2} \le ch_{\varepsilon}^{\frac{p-2\beta}{2p}}\left[\frac{|\mu|(Q_{4R})}{(4R)^{n}}\right]^{\frac{1}{2}}\left(\mean{Q_{4R}}(|Du|+s)^{q}\,dx\right)^{\frac{2-p}{2q}}.
\end{equation}

We note that
\[ \lim_{\varepsilon\rightarrow0}M_{\varepsilon} = 0 \;\; \Longrightarrow \;\; Du=Dw_{1} \text{ a.e. in }Q_{4R}, \]
and in this case there is nothing to prove. Hence, we may assume that $\inf_{\varepsilon}M_{\varepsilon} > 0$, which implies that there exists a constant $\varepsilon_{0} > 0$ such that $\varepsilon < (RM_{\varepsilon})^{p/(p-\beta)}$ whenever $\varepsilon \in (0,\varepsilon_{0})$. In turn, \eqref{def.he} gives
\begin{equation}\label{hepsilonest} 
h_{\varepsilon} < 2(RM_{\varepsilon})^{\frac{p}{p-\beta}} \qquad \forall \; \varepsilon \in (0,\varepsilon_{0}).
\end{equation}
With such a value of $\varepsilon$, we connect \eqref{I1.est}, \eqref{I2.est}, and \eqref{hepsilonest} to \eqref{est.I1.I2}, and then apply Young's inequality to have
\begin{align*}
\lefteqn{ M_{\varepsilon}  \le ch_{\varepsilon}^{\frac{1-\beta}{p}}\left[\frac{|\mu|(Q_{4R})}{(4R)^{n}}\right]^{\frac{1}{p}} + ch_{\varepsilon}^{\frac{p-2\beta}{2 p}}\left[\frac{|\mu|(Q_{4R})}{(4R)^{n}}\right]^{\frac{1}{2}}\left(\mean{Q_{4R}}(|Du|+s)^{q}\,dx\right)^{\frac{2-p}{2q}} } \\
& \le cM_{\varepsilon}^{\frac{1-\beta}{p-\beta}}R^{\frac{1-\beta}{p-\beta}}\left[\frac{|\mu|(Q_{4R})}{(4R)^{n}}\right]^{\frac{1}{p}}  + cM_{\varepsilon}^{\frac{p-2\beta}{2(p-\beta)}}R^{\frac{p-2\beta}{2(p-\beta)}}\left[\frac{|\mu|(Q_{4R})}{(4R)^{n}}\right]^{\frac{1}{2}}\left(\mean{Q_{4R}}(|Du|+s)^{q}\,dx\right)^{\frac{2-p}{2q}} \\
& \le \frac{1}{2}M_{\varepsilon} + cR^{\frac{1-\beta}{p-1}}\left[\frac{|\mu|(Q_{4R})}{(4R)^{n-1}}\right]^{\frac{p-\beta}{p(p-1)}} + cR^{\frac{p-2\beta}{p}}\left[\frac{|\mu|(Q_{4R})}{(4R)^{n-1}}\right]^{\frac{p-\beta}{p}}\left(\mean{Q_{4R}}(|Du|+s)^{q}\,dx\right)^{\frac{(p-\beta)(2-p)}{pq}}
\end{align*}
and therefore
\begin{equation*}
M_{\varepsilon}^{\frac{p}{p-\beta}} \le cR^{\frac{p(1-\beta)}{(p-\beta)(p-1)}}\left[\frac{|\mu|(Q_{4R})}{(4R)^{n}}\right]^{\frac{1}{p-1}} + cR^{\frac{p-2\beta}{p-\beta}}\left[\frac{|\mu|(Q_{4R})}{(4R)^{n}}\right]\left(\mean{Q_{4R}}(|Du|+s)^{q}\,dx\right)^{\frac{2-p}{q}}.
\end{equation*}
This with \eqref{Du-Dw.M} implies
\begin{align}\label{sdsdsd}
\lefteqn{ \left(\mean{Q_{4R}}|Du-Dw_{1}|^{q}\chi_{\mathcal{B}_{\varepsilon}}\,dx\right)^{\frac{1}{q}} \le cM_{\varepsilon}h_{\varepsilon}^{\frac{\beta}{p}} \overset{\eqref{hepsilonest}}{\le} cR^{\frac{\beta}{p-\beta}}M_{\varepsilon}^{\frac{p}{p-\beta}} } \nonumber \\
& \le c\left[\frac{|\mu|(Q_{4R})}{(4R)^{n-1}}\right]^{\frac{1}{p-1}} + c\left[\frac{|\mu|(Q_{4R})}{(4R)^{n-1}}\right]\left(\mean{Q_{4R}}(|Du|+s)^{q}\,dx\right)^{\frac{2-p}{q}}.
\end{align}
In a similar way, this time using \eqref{u-w.M}, we also have
\begin{align*}
\lefteqn{ \left(\mean{Q_{4R}}|u-w_{1}|^{\frac{\beta q}{(1-q)p}}\chi_{\mathcal{B}_{\varepsilon}}\,dx\right)^{\frac{(1-q)p}{\beta q}} \le ch_{\varepsilon} \overset{\eqref{hepsilonest}}{\le} c(RM_{\varepsilon})^{\frac{p}{p-\beta}} } \\
& \le cR^{\frac{p}{p-1}}\left[\frac{|\mu|(Q_{4R})}{(4R)^{n}}\right]^{\frac{1}{p-1}} + cR^{2}\left[\frac{|\mu|(Q_{4R})}{(4R)^{n}}\right]\left(\mean{Q_{4R}}(|Du|+s)^{q}\,dx\right)^{\frac{2-p}{q}}.
\end{align*}
Then H\"older's inequality and some elementary manipulations lead to
\begin{align}\label{szszsz}
\lefteqn{ \frac{1}{R}\left(\mean{Q_{4R}}|u-w_{1}|^{q}\chi_{\mathcal{B}_{\varepsilon}}\,dx\right)^{\frac{1}{q}} } \nonumber \\
& \le c\left[\frac{|\mu|(Q_{4R})}{(4R)^{n-1}}\right]^{\frac{1}{p-1}} + c\left[\frac{|\mu|(Q_{4R})}{(4R)^{n-1}}\right]\left(\mean{Q_{4R}}(|Du|+s)^{q}\,dx\right)^{\frac{2-p}{q}}.
\end{align}
Combining \eqref{sdsdsd} and \eqref{szszsz}, and then letting $\varepsilon \rightarrow 0$, we conclude with the desired estimate.
\end{proof}

\subsection{Comparison with \eqref{homogeneous} in the case $\eqref{divide.case}_{2}$}
In this case, the arguments in the proof of \cite[Lemma~2.5]{NP23ARMA} can be applied to $OP(\psi;\mu)$, which gives: 
\begin{lemma}
Let $u\in \mathcal{A}^{g}_{\psi}(\Omega) $ be the weak solution to \eqref{opmu} under assumptions \eqref{growth} with $\eqref{divide.case}_{2}$, and let $w_{1} \in \mathcal{A}^{u}_{\psi}(Q_{4R})$ be as in \eqref{homogeneous}. Then
\begin{align}\label{Du-Dw.mixed}
\lefteqn{ \left(\mean{Q_{4R}}|Du-Dw_{1}|^{\kappa}\,dx\right)^{\frac{1}{\kappa}} + \frac{1}{R}\left(\mean{Q_{4R}}|u-w_{1}|^{\kappa}\,dx\right)^{\frac{1}{\kappa}} } \nonumber \\
& \le c\left[\frac{|\mu|(Q_{4R})}{(4R)^{n-1}}\right]^{\frac{1}{p-1}} + c\left[\frac{|\mu|(Q_{4R})}{(4R)^{n-1}}\right]\avenorm{|Du|+s}_{L^{\frac{(p-1)(2-p)}{3-p}}_{x'}L^{2-p}_{x_{n}}(Q_{4R})}
\end{align}
holds for a constant $c\equiv c(\data)$, where $\kappa$ is as in \eqref{def.kappa}.

\end{lemma}
\begin{proof}
By using a standard scaling argument, we may assume that $Q_{4R} \equiv Q_{1}(0) \equiv Q_{1}$ and
\begin{equation*}
[|\mu|(Q_{1})]^{\frac{1}{p-1}} + [|\mu|(Q_{1})]\||Du|+s\|_{L^{\frac{(p-1)(2-p)}{3-p}}_{x'}L^{2-p}_{x_{n}}(Q_{1})}^{2-p} \le 1.
\end{equation*}
For any $k>0$, we recall the truncation operator $T_{k}$ given in \eqref{truncation.op}. 
Testing \eqref{opmu} and \eqref{homogeneous} with $\phi \equiv u + T_{2k}(w_{1}-u)$ and $\phi \equiv w_{1} - T_{2k}(w_{1}-u)$, respectively, we have
\[ \int_{Q_{1} \cap \{|u-w_{1}|<2k\}}|V(Du)-V(Dw_{1})|^{2}\,dx \le ck \]
for a constant $c \equiv c(\data)$. 
Then, by following the proof of \cite[Lemma~2.5]{NP23ARMA}, we have the desired estimate.
\end{proof}

\subsection{Reverse H\"older type inequalities for $OP(\psi;\mu)$}
To proceed further, we need certain reverse H\"older type inequalities for $Du$. 
Once we have \lemref{reverse.holder.w1}, \lemref{revhol.u}, \lemref{ccp.mixed} and the above two comparison estimates, we can obtain the following two lemmas, see \cite[Lemma~2.1]{NP23ME} and \cite[Lemma~2.6 and Remark~2.7]{NP23ARMA} for each case.
We note that \lemref{reverse.holder.w1}, \lemref{revhol.u} and \lemref{ccp.mixed} also hold in the case $p>2-1/n$, which along with \cite[Lemma~5.2]{BSY} give a new proof of \cite[Lemma~5.3]{BSY}.

\begin{lemma}
Let $u \in \mathcal{A}^{g}_{\psi}(\Omega)$ be the weak solution to \eqref{opmu} under assumptions \eqref{growth} with $\eqref{divide.case}_{1}$.
Then for any $q$ as in \eqref{qrange}, $\varepsilon \in (0,q]$ and $\sigma \in (0,1)$, we have
\begin{align}\label{revhol.u.1}
\left(\mean{Q_{\sigma r}}(|Du|+s)^{q}\,dx\right)^{\frac{1}{q}} & \le c\left(\mean{Q_{r}}(|Du|+s)^{\varepsilon}\,dx\right)^{\frac{1}{\varepsilon}} \nonumber \\
& \quad  + c\left[\frac{|\mu|(Q_{r})}{r^{n-1}}\right]^{\frac{1}{p-1}} + c\left(\mean{Q_{r}}\varphi^{*}(|A(D\psi)-A(\xi_{0})|)\,dx\right)^{\frac{1}{p'}}
\end{align}
for a constant $c\equiv c(\data,q,\varepsilon,\sigma)$, whenever $Q_{\sigma r} \subset Q_{r} \subset \Omega$ are concentric cubes and $\xi_{0} \in \mathbb{R}^{n}$.
\end{lemma}

\begin{lemma}
Let $u \in \mathcal{A}^{g}_{\psi}(\Omega)$ be the weak solution to \eqref{opmu} under assumptions \eqref{growth} with $\eqref{divide.case}_{2}$.
With $\kappa$ given in \eqref{def.kappa}, let 
\[ \theta \in \left(0,\frac{2\kappa(p-1)}{(2-p)(p-\kappa)} \right) \]
and define $s_{1}$ and $s_{2}$ by
\[ \frac{1}{2-p} = \frac{\theta}{\kappa} + \frac{1-\theta}{s_{1}}, \qquad \frac{3-p}{(p-1)(2-p)} = \frac{\theta}{\kappa} + \frac{1-\theta}{s_{2}}. \]
Then 
\[ 2-p < s_{1} < p, \quad s_{1} > s_{2} > \frac{(p-1)(2-p)}{3-p}, \quad s_{2} < \frac{s_{1}(n-1)(p-1)}{s_{1}(n-1) -p+1}. \]
Moreover, for any $\varepsilon \in (0,\kappa]$ and $\sigma \in (0,1)$, we have
\begin{align}\label{revhol.mixed}
\avenorm{|Du|+s}_{L^{s_{2}}_{x'}L^{s_{1}}_{x_{n}}(Q_{\sigma r})} & \le c\left(\mean{Q_{r}}(|Du|+s)^{\varepsilon}\,dx\right)^{\frac{1}{\varepsilon}} \nonumber \\
& \quad + c\left[\frac{|\mu|(Q_{r})}{r^{n-1}}\right]^{\frac{1}{p-1}} + c\left(\mean{Q_{r}}\varphi^{*}(|A(D\psi)-A(\xi_{0})|)\,dx\right)^{\frac{1}{p'}}
\end{align}
for a constant $c\equiv c(\data,s_{1},s_{2},\sigma,\varepsilon)$, whenever $Q_{\sigma r} \subset Q_{r} \subset \Omega$ are concentric cubes and $\xi_{0} \in \mathbb{R}^{n}$. 
\end{lemma}
From \eqref{1st.comparison}, \eqref{Du-Dw.mixed}, \eqref{revhol.u.1}, and \eqref{revhol.mixed}, we conclude with the following comparison estimate.
\begin{lemma}\label{Du-Dw1}
Let $u$ and $w_{1}$ be the weak solutions to \eqref{opmu} and \eqref{homogeneous}, respectively, under assumptions \eqref{growth} and \eqref{p.bound}. Then for any $q,\varepsilon \in (0,\kappa]$ and $\xi_{0} \in \mathbb{R}^{n}$, we have
\begin{align}\label{final.comparison}
\lefteqn{ \left(\mean{Q_{4R}}|Du-Dw_{1}|^{q}\,dx\right)^{\frac{1}{q}} + \frac{1}{R}\left(\mean{Q_{4R}}|u-w_{1}|^{q}\,dx\right)^{\frac{1}{q}} } \nonumber \\
& \le c\left[\frac{|\mu|(Q_{8R})}{(8R)^{n-1}}\right]^{\frac{1}{p-1}} + c\left[\frac{|\mu|(Q_{8R})}{(8R)^{n-1}}\right]\left(\mean{Q_{8R}}(|Du|+s)^{\varepsilon}\,dx\right)^{\frac{2-p}{\varepsilon}} \nonumber \\
& \quad + c\left[\frac{|\mu|(Q_{8R})}{(8R)^{n-1}}\right]\left(\mean{Q_{8R}}\varphi^{*}(|A(D\psi)-A(\xi_{0})|)\,dx\right)^{\frac{2-p}{p}}
\end{align}
for a constant $c\equiv c(\data,q,\varepsilon)$.
\end{lemma}

\subsection{Comparison with obstacle-free problems}
Next, we consider the two Dirichlet problems:
\begin{equation*}
\left\{
\begin{aligned}
 -\mathrm{div}\, A(Dw_{2}) &= -\mathrm{div}\, A(D\psi) &\textrm{in } & Q_{2R}, \\
 w_{2} & = w_{1} & \textrm{on } & \partial Q_{2R},
\end{aligned}
\right.
\end{equation*}
and 
\begin{equation}\label{homoeq}
\left\{
\begin{aligned}
 -\mathrm{div} \, A(Dv) & =0  & \textrm{in } & Q_{R}, \\
 v & = w_{2} & \textrm{on } & \partial Q_{R}.
\end{aligned}
\right.
\end{equation}
The following comparison estimate can be proved in a completely similar way as in \cite[Lemma~5.8]{BSY}, with the help of \eqref{exp.down}.
\begin{lemma}\label{w1.v.comparison}
Let $w_{1}$, $w_{2}$, and $v$ be defined as above, under assumptions \eqref{growth} with $p>1$. Then we have
\begin{align*}
\mean{Q_{R}}|V(Dw_{1})-V(Dv)|^{2}\,dx
& \le \varepsilon (\varphi_{|z_{0}|})^{*}\left[\left(\mean{Q_{4R}}|A(Dw_{1})-A(z_{0})|^{\sigma}\,dx\right)^{\frac{1}{\sigma}}\right] \nonumber \\
& \quad + c\varepsilon^{1-\max\{p,2\}}\mean{Q_{4R}}(\varphi_{|z_{0}|})^{*}(|A(D\psi)-A(\xi_{0})|)\,dx
\end{align*}
for a constant $c \equiv c(\data)$, whenever $z_{0}, \xi_{0} \in \mathbb{R}^{n}$ and $\varepsilon, \sigma \in (0,1]$.
\end{lemma}

We then establish a comparison estimate between $A(Dw_{1})$ and $A(Dv)$.

\begin{lemma}\label{lin.comp.w1.v}
Let $w_{1}$ and $v$ be as in \eqref{homogeneous} and \eqref{homoeq}, respectively, under assumptions \eqref{growth} with $1<p \le 2$. Then, with $\kappa$ given in \eqref{def.kappa}, we have
\begin{align}\label{ADw1.ADv}
\lefteqn{ \left(\mean{Q_{R}}|A(Dw_{1})-A(Dv)|^{\kappa}\,dx\right)^{\frac{1}{\kappa}}} \nonumber \\
& \le \varepsilon\left(\mean{Q_{4R}}|A(Dw_{1})-A(z_{0})|^{\kappa}\,dx\right)^{\frac{1}{\kappa}} + c_{\varepsilon} \left(\mean{Q_{4R}}\varphi^{*}(|A(D\psi)-A(\xi_{0})|)\,dx\right)^{\frac{1}{p'}}
\end{align}
for any $\varepsilon \in (0,1)$ and $z_{0},\xi_{0}\in\mathbb{R}^{n}$, where $c_{\varepsilon} \equiv  c_{\varepsilon}(\data,\varepsilon)$ is proportional to some negative power of $\varepsilon$. 
\end{lemma}
\begin{proof}
We first estimate
\begin{align*}
\lefteqn{ (\varphi_{|z_{0}|})^{*}\left[\left(\mean{Q_{R}}|A(Dw_{1})-A(Dv)|^{\kappa}\,dx\right)^{\frac{1}{\kappa}}\right] \le \mean{Q_{R}}(\varphi_{|z_{0}|})^{*}(|A(Dw_{1})-A(Dv)|)\,dx }\\
\overset{\eqref{shift.change}}&{\le} c\gamma_{1}^{-1}\mean{Q_{R}}(\varphi_{|Dw_{1}|})^{*}(|A(Dw_{1})-A(Dv)|)\,dx + \gamma_{1}\mean{Q_{R}}|V(Dw_{1})-V(z_{0})|^{2}\,dx \\
\overset{\eqref{mono.shift}}&{\le} c\gamma_{1}^{-1}\mean{Q_{R}}|V(Dw_{1})-V(Dv)|^{2}\,dx + \gamma_{1}\mean{Q_{R}}|V(Dw_{1})-V(z_{0})|^{2}\,dx
\end{align*}
for any $\gamma_{1} \in (0,1)$.
We then apply \lemref{w1.v.comparison} and \lemref{reverse.holder.w1} to estimate each term on the right-hand side, thereby obtaining
\begin{align*}
\lefteqn{ (\varphi_{|z_{0}|})^{*}\left[\left(\mean{Q_{R}}|A(Dw_{1})-A(Dv)|^{\kappa}\,dx\right)^{\frac{1}{\kappa}}\right] }\\
& \le c\gamma_{1}^{-1}\gamma_{2}(\varphi_{|W_{R}|})^{*}\left[\left(\mean{Q_{4R}}|A(Dw_{1})-A(z_{0})|^{\kappa}\,dx\right)^{\frac{1}{\kappa}}\right] \\ 
& \quad+ c\gamma_{1}^{-1}\gamma_{2}^{-1}\mean{Q_{4R}}(\varphi_{|z_{0}|})^{*}(|A(D\psi)-A(\xi_{0})|)\,dx \\
& \quad + c\gamma_{1}(\varphi_{|z_{0}|})^{*}\left[\left(\mean{Q_{2R}}|A(Dw_{1})-A(z_{0})|^{\kappa}\,dx \right)^{\frac{1}{\kappa}}\right] \\
& \quad + c\gamma_{1}\mean{Q_{2R}}(\varphi_{|z_{0}|})^{*}(|A(D\psi)-A(\xi_{0})|)\,dx
\end{align*}
for any $\gamma_{2} \in (0,1)$. Choosing $\gamma_{2} = \gamma_{1}^{2}$, we arrive at
\begin{align*}
(\varphi_{|z_{0}|})^{*}\left[\left(\mean{Q_{R}}|A(Dw_{1})-A(Dv)|^{\kappa}\,dx\right)^{\frac{1}{\kappa}}\right] 
& \le c\gamma_{1}(\varphi_{|z_{0}|})^{*}\left[\left(\mean{Q_{4R}}|A(Dw_{1})-A(z_{0})|^{\kappa}\,dx\right)^{\frac{1}{\kappa}}\right] \\
& \quad + c\gamma_{1}^{-3}\mean{Q_{4R}}(\varphi_{|z_{0}|})^{*}(|A(D\psi)-A(\xi_{0})|)\,dx.
\end{align*}
Finally, in the proof of \cite[Lemma 5.8]{BSY}, it is shown that $t \mapsto [((\varphi_{|z_{0}|})^{*})^{-1}(t)]^{p'}$ is quasi-convex. Therefore, with a suitable choice of $\gamma_{1}$, we apply Jensen's inequality to the last term and then use the fact that $t^{p'} \leq c \varphi^{*}(t)$ for any $1<p\le2$ and some $c=c(p)$, in order to conclude with \eqref{ADw1.ADv}.
\end{proof}

\section{Comparison estimates under alternatives}\label{sec.lin.comp}
In this section, we linearize the comparison estimates between \eqref{opmu} and \eqref{homoeq} established in the previous section. 
Throughout this section, we keep assuming \eqref{regular} to ensure the existence of weak solutions to \eqref{opmu}.
We then fix a cube
\begin{equation}\label{ball.parameter}
Q_{4MR} \equiv Q_{4MR}(x_{0}) \Subset \Omega \quad \textrm{with} \quad M\ge8 \quad \textrm{and} \quad R\le1,
\end{equation}
where $M$ is a free parameter whose relevant value will be determined later in this section.

\subsection{The two-scales degenerate alternative}
We first consider the case when
\begin{equation}\label{deg}
\left(\mean{Q_{4MR}}|A(Du)-\mathcal{P}_{\kappa,Q_{4MR}}(A(Du))
|^{\kappa}\,dx\right)^{\frac{1}{\kappa}} \ge \theta\left[ |\mathcal{P}_{\kappa,Q_{R/M}}(A(Du))| + s^{p-1} \right]
\end{equation}
holds for another free parameter $\theta \in (0,1)$, where $\kappa$ and $M$ are as in \eqref{def.kappa} and \eqref{ball.parameter}, respectively.
The values of $M$ and $\theta$ will be determined in the next section, and their specific values do not affect the results in this section.

We observe that
\begin{align}\label{jj}
\lefteqn{\left(\mean{Q_{8R}}(|Du|+s)^{(p-1)\kappa}\,dx\right)^{\frac{1}{\kappa}}
 \overset{\eqref{a.diff}}{\le} c\left(\mean{Q_{8R}}(|A(Du)|+s^{p-1})^{\kappa}\,dx\right)^{\frac{1}{\kappa}} } \nonumber \\
& \le c\left(\mean{Q_{8R}}|A(Du)-\mathcal{P}_{\kappa,Q_{R/M}}(A(Du))|^{\kappa}\,dx\right)^{\frac{1}{\kappa}} + c\left[|\mathcal{P}_{\kappa,Q_{R/M}}(A(Du))|+s^{p-1}\right] \nonumber \\
& \le cM^{\frac{2n}{\kappa}}\left(\mean{Q_{4MR}}|A(Du)-\mathcal{P}_{\kappa,Q_{4MR}}(A(Du))
|^{\kappa}\,dx\right)^{\frac{1}{\kappa}} + c\left[|\mathcal{P}_{\kappa,Q_{R/M}}(A(Du))|+s^{p-1}\right] \nonumber \\
\overset{\eqref{deg}}&{\le} cM^{\frac{2n}{\kappa}}\left(1+\frac{1}{\theta}\right)\left(\mean{Q_{4MR}}|A(Du)-\mathcal{P}_{\kappa,Q_{4MR}}(A(Du))
|^{\kappa}\,dx\right)^{\frac{1}{\kappa}} 
\end{align}
holds for a constant $c \equiv c(\data)$. Using this, we establish the following comparison estimate.
\begin{lemma}\label{a.comp.deg}
Let $\theta \in (0,1)$ be such that \eqref{deg} holds with $M \ge 8$ as in \eqref{ball.parameter}. Then we have
\begin{align}\label{lincomp.homo.deg}
\lefteqn{ \left(\mean{Q_{R}}|A(Du)-A(Dv)|^{\kappa}\,dx\right)^{\frac{1}{\kappa}} } \nonumber \\
& \le \varepsilon M^{\frac{2n}{\kappa}}\left(1+\frac{1}{\theta}\right)\left(\mean{Q_{4MR}}|A(Du)-\mathcal{P}_{\kappa,Q_{4MR}}(A(Du))
|^{\kappa}\,dx\right)^{\frac{1}{\kappa}} \nonumber \\
& \quad + c_{\varepsilon}\left[\frac{|\mu|(Q_{8R})}{(8R)^{n-1}}\right] + c_{\varepsilon}\left(\mean{Q_{8R}}\varphi^{*}(|A(D\psi)-A(\xi_{0})|)\,dx\right)^{\frac{1}{p'}}
\end{align}
for any $\xi_{0} \in \mathbb{R}^{n}$ and $\varepsilon \in (0,1]$, where $c_{\varepsilon} \equiv c_{\varepsilon}(\data,\varepsilon)$ is proportional to some negative power of $\varepsilon$.
\end{lemma}
\begin{proof}
We use \eqref{final.comparison} and Young's inequality to have
\begin{align*}
\lefteqn{ \left(\mean{Q_{4R}}|A(Du)-A(Dw_{1})|^{\kappa}\,dx\right)^{\frac{1}{\kappa}}
\overset{\eqref{a.diff}}{\le} c\left(\mean{Q_{4R}}|Du-Dw_{1}|^{(p-1)\kappa}\,dx\right)^{\frac{1}{\kappa}} }\\
& \le c\left[\frac{|\mu|(Q_{8R})}{(8R)^{n-1}}\right] + c\left[\frac{|\mu|(Q_{8R})}{(8R)^{n-1}}\right]^{p-1}\left(\mean{Q_{8R}}(|Du|+s)^{(p-1)\kappa}\,dx\right)^{\frac{2-p}{\kappa}} \\
& \le \varepsilon \left(\mean{Q_{8R}}(|Du|+s)^{(p-1)\kappa}\,dx\right)^{\frac{1}{\kappa}} + c\varepsilon^{\frac{p-2}{p-1}}\left[\frac{|\mu|(Q_{8R})}{(8R)^{n-1}}\right]
\end{align*}
for any $\varepsilon \in (0,1]$. Combining this estimate with \eqref{ADw1.ADv} and using \eqref{jj}, we obtain \eqref{lincomp.homo.deg}.
\end{proof}

\subsection{The two-scales non-degenerate alternative}
Here we consider the case when \eqref{deg} fails, namely
\begin{equation}\label{ndeg}
\left(\mean{Q_{4MR}}|A(Du)-\mathcal{P}_{\kappa,Q_{4MR}}(A(Du))
|^{\kappa}\,dx\right)^{\frac{1}{\kappa}} < \theta\left[ |\mathcal{P}_{\kappa,Q_{R/M}}(A(Du))| + s^{p-1} \right]
\end{equation}
holds for a number $\theta \in (0,1)$. In the following, we denote
\begin{equation}\label{lambda.choice}
\lambda \coloneqq \left|\mathcal{P}_{\kappa,Q_{R/M}}(A(Du))\right|^{\frac{1}{p-1}} + s.
\end{equation}
Then we have the following:
\begin{lemma}\label{scale.change}
Let $\lambda$ be as in \eqref{lambda.choice}. For every $M \ge 8$ as in \eqref{ball.parameter}, there exists a number $\theta \equiv \theta(n,M)$ such that if \eqref{ndeg} is in force, then
\begin{equation}\label{mean.bound}
\left(\mean{Q_{\sigma R}}(|Du|+s)^{(p-1)\kappa}\,dx\right)^{\frac{1}{\kappa}} \le c\lambda^{p-1},\quad \forall\,\,\sigma \in [1/M,4M]
\end{equation}
holds for a constant $c \equiv c(\data)$.
\end{lemma}
\begin{proof}
Using \eqref{ndeg}, we have
\begin{align*}
\lefteqn{ \left(\mean{Q_{\sigma R}}(|Du|+s)^{(p-1)\kappa}\,dx\right)^{\frac{1}{\kappa}} \overset{\eqref{a.diff}}{\le} c\left(\mean{Q_{\sigma R}}(|A(Du)|+s^{p-1})^{\kappa}\,dx\right)^{\frac{1}{\kappa}} } \\
& \le c\left(\mean{Q_{\sigma R}}|A(Du)-\mathcal{P}_{\kappa,Q_{4MR}}(A(Du))
|^{\kappa}\,dx\right)^{\frac{1}{\kappa}} + c\left|\mathcal{P}_{\kappa,Q_{4MR}}(A(Du))
-\mathcal{P}_{\kappa,Q_{R/M}}(A(Du))\right| \\
& \quad + c\left[|\mathcal{P}_{\kappa,Q_{R/M}}(A(Du))| + s^{p-1}\right] \\
& \le c\left(\mean{Q_{\sigma R}}|A(Du)-\mathcal{P}_{\kappa,Q_{4MR}}(A(Du))
|^{\kappa}\,dx\right)^{\frac{1}{\kappa}} + c\left(\mean{Q_{R/M}}|A(Du)-\mathcal{P}_{\kappa,Q_{4MR}}(A(Du))
|^{\kappa}\,dx\right)^{\frac{1}{\kappa}} \\
& \quad + c\left(\mean{Q_{R/M}}|A(Du)-\mathcal{P}_{\kappa,Q_{R/M}}(A(Du))|^{\kappa}\,dx\right)^{\frac{1}{\kappa}} + c\left[|\mathcal{P}_{\kappa,Q_{R/M}}(A(Du))|+s^{p-1}\right] \\
& \le c\left[\left(\frac{M}{\sigma}\right)^{n}+M^{2n}\right]^{\frac{1}{\kappa}}\left(\mean{Q_{4MR}}|A(Du)-\mathcal{P}_{\kappa,Q_{4MR}}(A(Du))
|^{\kappa}\,dx\right)^{\frac{1}{\kappa}} \\
& \quad + c\left[|\mathcal{P}_{\kappa,Q_{R/M}}(A(Du))|+s^{p-1}\right] \\
& \le c(1+M^{2n}\theta)^{\frac{1}{\kappa}}\left[|\mathcal{P}_{\kappa,Q_{R/M}}(A(Du))|+s^{p-1}\right].
\end{align*}
Then we choose the constant $\theta$ so small that 
\begin{equation}\label{theta.M}
M^{2n}\theta \le 1
\end{equation}
in order to conclude with \eqref{mean.bound}.
\end{proof}

We now prove a counterpart of \lemref{a.comp.deg} after fixing the values of $\theta$ and $M$.

\begin{lemma}\label{a.comp.ndeg}
It is possible to determine $\theta$ and $M$ as functions of $\data$ such that if \eqref{ndeg} is in force, then there holds
\begin{align}\label{lincomp.homo.ndeg}
\lefteqn{ \left(\mean{Q_{R/M}}|A(Du)-A(Dv)|^{\kappa}\,dx\right)^{\frac{1}{\kappa}} } \nonumber \\ 
& \le c\left[\frac{|\mu|(Q_{4MR})}{(4MR)^{n-1}}\right] + c\left(\mean{Q_{4MR}}\varphi^{*}(|A(D\psi)-(A(D\psi))_{Q_{4MR}}|)\,dx\right)^{\frac{1}{p'}}
\end{align}
for a constant $c \equiv c(\data)$.
\end{lemma}

In the proof of \lemref{a.comp.ndeg}, we will distinguish two cases, making use of another free parameter $\sigma_{1} \in (0,1)$. The first one is when the following inequality holds:
\begin{equation}\label{measure.small}
\left[\frac{|\mu|(Q_{4MR})}{(4MR)^{n-1}}\right] + \left(\mean{Q_{4MR}}\varphi^{*}(|A(D\psi) - (A(D\psi))_{Q_{4MR}}|)\,dx\right)^{\frac{1}{p'}} \le \sigma_{1}\lambda^{p-1}.
\end{equation}
The second one is when the above inequality fails; that is,
\begin{equation}\label{measure.large}
\lambda^{p-1} < \frac{1}{\sigma_{1}}\left[\frac{|\mu|(Q_{4MR})}{(4MR)^{n-1}}\right] + \frac{1}{\sigma_{1}}\left(\mean{Q_{4MR}}\varphi^{*}(|A(D\psi)-(A(D\psi))_{Q_{4MR}}|)\,dx\right)^{\frac{1}{p'}}.
\end{equation}
The value of $\sigma_{1}$ will be determined in \lemref{grad.bound.lem} below.

\subsubsection{Proof of \lemref{a.comp.ndeg} in the first case \eqref{measure.small} and determination of $\sigma_{1}$}

\begin{lemma}\label{grad.bound.lem}
There exists a choice of the parameters
\begin{equation*}
M \equiv M(\data) \ge 8 \qquad \text{and} \qquad \sigma_{1} \equiv \sigma_{1}(\data,M) \in (0,1)
\end{equation*}
such that, if $\theta \equiv \theta(n,M)$ is the constant determined in \lemref{scale.change} and \eqref{ndeg} is in force, then the following bounds hold:
\begin{equation}\label{v.bound}
\frac{\lambda}{c} \le |Dv|+s \;\; \textrm{in } Q_{4R/M} \quad \textrm{and} \quad |Dv|+s \le c \lambda \;\; \textrm{in } Q_{R/2},
\end{equation}
with constants $ c$ depending only on $\data$.
\end{lemma}

\begin{proof}
We first prove the upper bound. Using \lemref{grad.hol.v} and \lemref{lin.comp.w1.v}, we have
\begin{align*}
\left[\sup_{Q_{R/2}}(|Dv|+s)\right]^{(p-1)\kappa} \overset{\eqref{a.diff}}&{\le} c\mean{Q_{R}}(|A(Dv)|+s^{p-1})^{\kappa}\,dx \\
& \le c\mean{Q_{R}}(|A(Dw_{1})|+s^{p-1})^{\kappa}\,dx + c\mean{Q_{R}}|A(Dw_{1})-A(Dv)|^{\kappa}\,dx \\
& \le c\mean{Q_{4R}}(|A(Dw_{1})|+s^{p-1})^{\kappa}\,dx \\
& \quad + cM^{\frac{n\kappa}{p'}}\left(\mean{Q_{4MR}}\varphi^{*}(|A(D\psi)-(A(D\psi))_{Q_{4MR}}|)\,dx\right)^{\frac{\kappa}{p'}}.
\end{align*}
We then apply \eqref{final.comparison}, \eqref{mean.bound}, and \eqref{measure.small} in order to estimate
\begin{align}\label{ADw.mean}
\lefteqn{ \mean{Q_{4R}}(|A(Dw_{1})|+s^{p-1})^{\kappa}\,dx \overset{\eqref{a.diff}}{\le} \mean{Q_{4R}}(|Dw_{1}|+s)^{(p-1)\kappa}\,dx } \nonumber \\ 
& \le c\mean{Q_{4R}}(|Du|+s)^{(p-1)\kappa}\,dx + c\mean{Q_{4R}}|Du-Dw_{1}|^{(p-1)\kappa}\,dx \nonumber \\
& \le c\lambda^{(p-1)\kappa} + cM^{(n-1)\kappa}\left[\frac{|\mu|(Q_{4MR})}{(4MR)^{n-1}}\right]^{\kappa} \nonumber \\ 
& \quad + cM^{(n-1)(p-1)\kappa}\left[\frac{|\mu|(Q_{4MR})}{(4MR)^{n-1}}\right]^{(p-1)\kappa}\left(\mean{Q_{4R}}(|Du|+s)^{(p-1)\kappa}\,dx\right)^{\frac{2-p}{\kappa}} \nonumber \\
& \quad + cM^{\frac{2n-p}{p'}\kappa}\left[\frac{|\mu|(Q_{4MR})}{(4MR)^{n-1}}\right]^{(p-1)\kappa}\left(\mean{Q_{4MR}}\varphi^{*}(|A(D\psi)-(A(D\psi))_{Q_{4MR}}|)\,dx\right)^{\frac{(2-p)\kappa}{p'}} \nonumber \\
& \le c\left[ 1 + M^{n-1}\sigma_{1} + \left(M^{n-1}\sigma_{1}\right)^{p-1} + M^{\frac{2n-p}{p'}}\sigma_{1} \right]^{\kappa}\lambda^{(p-1)\kappa}.
\end{align}
Combining the above two estimates and using \eqref{measure.small}, we arrive at
\begin{align*}
\left[\sup_{Q_{R/2}}(|Dv|+s)\right]^{(p-1)\kappa} & \le c\left[1+M^{n-1}\sigma_{1} + \left(M^{n-1}\sigma_{1}\right)^{p-1} + M^{\frac{2n-p}{p'}}\sigma_{1} + M^{\frac{n}{p'}}\sigma_{1}\right]^{\kappa}\lambda^{(p-1)\kappa}
\end{align*}
for a constant $c \equiv c(\data)$.
By choosing $\sigma_{1} \equiv \sigma_{1}(\data,M)$ such that
\begin{equation}\label{sigma.cond2}
 M^{n-1}\sigma_{1} + M^{\frac{n}{p'}}\sigma_{1} + M^{\frac{2n-p}{p'}}\sigma_{1} \le 1,
\end{equation}
we conclude that
\begin{equation}\label{Dv.upper}
    \sup_{Q_{R/2}}(|Dv|+s) \le c\lambda
\end{equation}
holds with $c \equiv c(\data)$.

To prove the lower bound, 
By using \eqref{mean.bound}, we fix a constant $\csso \equiv \csso (\data)>1$ satisfying
\begin{equation*}
\frac{\lambda^{(p-1)\kappa}}{\csso } \le (|A(Du)|^{\kappa})_{Q_{4R/M}} + s^{(p-1)\kappa} \le \csso \lambda^{(p-1)\kappa}
\end{equation*}
to find
\begin{align}\label{ADv.mean.lower}
(|A(Dv)|^{\kappa})_{Q_{4R/M}} + s^{(p-1)\kappa}
& \ge (|A(Du)|^{\kappa})_{Q_{4R/M}} + s^{(p-1)\kappa} - (|A(Du)-A(Dv)|^{\kappa})_{Q_{4R/M}} \nonumber \\
& \ge \frac{\lambda^{(p-1)\kappa}}{\csso } - \mean{Q_{4R/M}}|A(Du)-A(Dv)|^{\kappa}\,dx,
\end{align}
where we have used the fact that $\kappa \in (0,1)$. In order to estimate the last integral, we split as follows:
\begin{align}\label{i1i2}
\lefteqn{ \mean{Q_{4R/M}}|A(Dv)-A(Du)|^{\kappa}\,dx } \nonumber \\
& \le cM^{n}\mean{Q_{R}}|A(Dv)-A(Dw_{1})|^{\kappa}\,dx + cM^{n}\mean{Q_{4R}}|A(Dw_{1})-A(Du)|^{\kappa}\,dx \nonumber \\
& \eqqcolon I_{1} + I_{2}.
\end{align}
We estimate $I_{2}$ as
\begin{align*}
\lefteqn{ I_{2} \overset{\eqref{a.diff}}{\le} cM^{n}\mean{Q_{4R}}|Dw_{1}-Du|^{(p-1)\kappa}\,dx } \\
\overset{\eqref{final.comparison}}&{\le} cM^{n+(n-1)\kappa}\left[\frac{|\mu|(Q_{4MR})}{(4MR)^{n-1}}\right]^{\kappa} \\
& \qquad + cM^{n+(n-1)(p-1)\kappa}\left[\frac{|\mu|(Q_{4MR})}{(4MR)^{n-1}}\right]^{(p-1)\kappa}\left(\mean{Q_{4R}}(|Du|+s)^{(p-1)\kappa}\,dx\right)^{2-p} \\
& \qquad + cM^{n+\frac{(2n-p)\kappa}{p'}}\left[\frac{|\mu|(Q_{4MR})}{(4MR)^{n-1}}\right]^{(p-1)\kappa}\left(\mean{Q_{4MR}}\varphi^{*}(|A(D\psi)-(A(D\psi))_{Q_{4MR}}|)\,dx\right)^{\frac{(2-p)\kappa}{p'}} \\
\overset{\eqref{measure.small}}&{\le} \cssot\left[M^{\frac{n}{\kappa}+n-1}\sigma_{1} + M^{\frac{n}{\kappa}+(n-1)(p-1)}\sigma_{1}^{p-1} + M^{\frac{n}{\kappa}+\frac{2n-p}{p'}}\sigma_{1}\right]^{\kappa}\lambda^{(p-1)\kappa}
\end{align*}
for a constant $\cssot \equiv \cssot (\data)$. Choosing $\sigma_{1} \equiv \sigma_{1}(\data,M)$ such that
\begin{equation}\label{sigma.cond4}
    \cssot \left[M^{\frac{n}{\kappa}+n-1}\sigma_{1} + M^{\frac{n}{\kappa}+(n-1)(p-1)}\sigma_{1}^{p-1} + M^{\frac{n}{\kappa}+\frac{2n-p}{p'}}\sigma_{1}\right]^{\kappa} \le \frac{1}{4\csso},
\end{equation}
we arrive at
\begin{equation}\label{itwo}
 I_{2} \le \frac{\lambda^{(p-1)\kappa}}{4\csso }.
\end{equation}
As for $I_{1}$, we have
\begin{align*}
I_{1} \overset{\eqref{ADw1.ADv}}&{\le} cM^{n}\varepsilon\mean{Q_{4R}}|A(Dw_{1})|^{\kappa}\,dx  + c_{\varepsilon} M^{n} \left(\mean{Q_{4R}}\varphi^{*}(|A(D\psi)-(A(D\psi))_{Q_{4MR}}|)\,dx\right)^{\frac{\kappa}{p'}}  \\
\overset{\eqref{ADw.mean},\eqref{sigma.cond2}}&{\le} cM^{n}\varepsilon\lambda^{(p-1)\kappa}
+ c_{\varepsilon} M^{n\left(1+\frac{\kappa}{p'}\right)}\left(\mean{Q_{4MR}}\varphi^{*}(|A(D\psi)-(A(D\psi))_{Q_{4MR}}|)\,dx\right)^{\frac{\kappa}{p'}} \\
\overset{\eqref{measure.small}}&{\le} \cssoth\left[M^{n}\varepsilon + c_{\varepsilon} M^{n\left(1+\frac{\kappa}{p'}\right)}\sigma_{1}^{\kappa}\right]\lambda^{(p-1)\kappa}
\end{align*}
for some constants $\cssoth \equiv \cssoth (\data)$ and $c_{\varepsilon} \equiv c_{\varepsilon}(\data,\varepsilon)$, whenever $\varepsilon \in (0,1)$.
Choosing $\varepsilon = 1/(8M^{n}\cssoth \csso )$ and then $\sigma_{1} \equiv \sigma_{1}(\data,M)$ satisfying
\begin{equation}\label{sigma.cond5}
c_{\varepsilon} \cssoth M^{n\left(1+\frac{\kappa}{p'}\right)} \sigma_{1}^{\kappa} \le \frac{1}{8\csso },
\end{equation}
it follows that
\begin{equation}\label{ione}
I_{1} \le \frac{\lambda^{(p-1)\kappa}}{4\csso }.
\end{equation}
Connecting \eqref{i1i2}, \eqref{itwo} and \eqref{ione} to \eqref{ADv.mean.lower}, we have
\begin{equation*}
(|A(Dv)|^{\kappa})_{Q_{4R/M}} + s^{(p-1)\kappa} \ge \frac{\lambda^{(p-1)\kappa}}{2\csso }.
\end{equation*}
We now choose a point $x_{0} \in Q_{4R/M}$ satisfying
\begin{equation}\label{1pt.bound2}
    |A(Dv(x_{0}))|^{\kappa} + s^{(p-1)\kappa} \ge \frac{\lambda^{(p-1)\kappa}}{2\csso }.
\end{equation}
Then, using the oscillation estimate \cite[Corollary 4.5]{BSY}, \eqref{mean.min} and \eqref{Dv.upper}, we find that
\begin{equation*}
    \osc_{Q_{4R/M}}A(Dv) \le \frac{c}{M^{\alpha_{A}}}\mean{Q_{R/2}}|A(Dv)|\,dx \le \frac{\cssf }{M^{\alpha_{A}}}\lambda^{p-1}
\end{equation*}
holds for a constant $\cssf \equiv \cssf (\data)$. Choosing $M$ such that
\begin{equation}\label{M.cond2}
    \frac{\cssf }{M^{\alpha_{A}}} \le \left(\frac{1}{4\csso }\right)^{\frac{1}{\kappa}}
\end{equation}
and then combining the resulting inequality with \eqref{1pt.bound2}, we obtain the lower bound in \eqref{v.bound} with some constant $c \equiv c(\data)$.
\end{proof}

\begin{remark}
The process of fixing the constants $\theta$, $M$ and $\sigma_{1}$ can be summarized as follows. We first fix $M \equiv M(\data)$ as in \lemref{grad.bound.lem} satisfying \eqref{M.cond2}. Then, by \lemref{scale.change}, we choose $\theta \equiv \theta(\data)$ such that \eqref{theta.M} holds. In a similar way, we finally determine $\sigma_{1} \equiv \sigma_{1}(\data)$ as in \lemref{grad.bound.lem}, by requiring that \eqref{sigma.cond2}, \eqref{sigma.cond4} and \eqref{sigma.cond5} are satisfied. Consequently, we have fixed all the parameters $\theta$, $M$ and $\sigma_{1}$ as universal constants depending only on $\data$, for which the assertions of \lemref{scale.change} and \lemref{grad.bound.lem} hold simultaneously. These values of the parameters will be used in the rest of the paper.
\end{remark}

We now prove estimate \eqref{lincomp.homo.ndeg}. We have
\begin{align}\label{ADu-ADw}
\lefteqn{ \left(\mean{Q_{R/M}}|A(Du)-A(Dv)|^{\kappa}\,dx\right)^{\frac{1}{\kappa}} } \nonumber \\ \overset{\eqref{a.diff}}&{\le} c\left(\mean{Q_{R/M}}(|Du|+|Dv|+s)^{(p-2)\kappa}|Du-Dv|^{\kappa}\,dx\right)^{\frac{1}{\kappa}} \nonumber \\
\overset{p<2}&{\le} c\left[\inf_{Q_{R/M}}(|Dv|+s)\right]^{p-2}\left(\mean{Q_{R/M}}|Du-Dv|^{\kappa}\,dx\right)^{\frac{1}{\kappa}} \nonumber \\
\overset{\eqref{v.bound}}&{\le} c\lambda^{p-2}\left(\mean{Q_{4R}}|Du-Dw_{1}|^{\kappa}\,dx + \mean{Q_{R}}|Dw_{1}-Dv|^{\kappa}\,dx\right)^{\frac{1}{\kappa}}.
\end{align}
We now estimate the two integrals in the right-hand side of \eqref{ADu-ADw}. We estimate the first one as
\begin{align}\label{ADu.ADw.1st}
\lefteqn{ \lambda^{p-2}\left(\mean{Q_{4R}}|Du-Dw_{1}|^{\kappa}\,dx\right)^{\frac{1}{\kappa}} } \nonumber \\
\overset{\eqref{final.comparison}}&{\le} c\lambda^{p-2}\left[\frac{|\mu|(Q_{8R})}{(8R)^{n-1}}\right]^{\frac{1}{p-1}} + c\lambda^{p-2}\left[\frac{|\mu|(Q_{8R})}{(8R)^{n-1}}\right]\left(\mean{Q_{8R}}(|Du|+s)^{(p-1)\kappa}\,dx\right)^{\frac{2-p}{p-1}} \nonumber \\
& \qquad + c\lambda^{p-2}\left[\frac{|\mu|(Q_{8R})}{(8R)^{n-1}}\right]\left(\mean{Q_{8R}}\varphi^{*}(|A(D\psi)-(A(D\psi))_{Q_{8R}}|)\,dx\right)^{\frac{2-p}{p}} \nonumber \\
\overset{\eqref{measure.small}}&{\le} c\left[\frac{|\mu|(Q_{4MR})}{(4MR)^{n-1}}\right].
\end{align}
The second one is estimated by applying H\"older's inequality and then \cite[(6.34)-(6.36)]{BSY}:
\begin{align}\label{ADu.ADw.2nd}
\lambda^{p-2}\left(\mean{Q_{R}}|Dw_{1}-Dv|^{\kappa}\,dx\right)^{\frac{1}{\kappa}} & \le \lambda^{p-2}\mean{Q_{R}}|Dw_{1}-Dv|\,dx \nonumber \\
& \le c\left(\mean{Q_{4MR}}\varphi^{*}(|A(D\psi)-(A(D\psi))_{Q_{4MR}}|)\,dx\right)^{\frac{1}{p'}}
\end{align}
Combining \eqref{ADu-ADw}, \eqref{ADu.ADw.1st}, and \eqref{ADu.ADw.2nd}, we obtain the desired estimate \eqref{lincomp.homo.ndeg}.

\subsubsection{Proof of \lemref{a.comp.ndeg} in the second case \eqref{measure.large}}
We observe that, from \eqref{mean.bound} and \eqref{measure.large},
\begin{align}\label{lambda.bound}
\lefteqn{ \left(\mean{Q_{\sigma R}}(|Du|+s)^{(p-1)\kappa}\,dx\right)^{\frac{1}{\kappa}} } \nonumber \\
& \le c\left[\frac{|\mu|(Q_{4MR})}{(4MR)^{n-1}}\right] + c\left(\mean{Q_{4MR}}\varphi^{*}(|A(D\psi)-(A(D\psi))_{Q_{4MR}}|)\,dx\right)^{\frac{1}{p'}}
\end{align}
holds whenever $\sigma \in [1/M,4M]$, where $c \equiv c(\data)$. 

Now we prove \eqref{lincomp.homo.ndeg}. We have
\begin{align}\label{uw}
\lefteqn{ \left(\mean{Q_{R/M}}|A(Du)-A(Dw_{1})|^{\kappa}\,dx\right)^{\frac{1}{\kappa}} \overset{\eqref{a.diff}}{\le} cM^{\frac{n}{\kappa}}\left(\mean{Q_{4R}}|Du-Dw_{1}|^{(p-1)\kappa}\,dx\right)^{\frac{1}{\kappa}} } \nonumber \\
\overset{\eqref{final.comparison}}&{\le} c\left[\frac{|\mu|(Q_{8R})}{(8R)^{n-1}}\right] + c\left[\frac{|\mu|(Q_{8R})}{(8R)^{n-1}}\right]^{p-1}\left(\mean{Q_{8R}}(|Du|+s)^{(p-1)\kappa}\,dx\right)^{\frac{2-p}{\kappa}} \nonumber \\
& \qquad + c\left[\frac{|\mu|(Q_{8R})}{(8R)^{n-1}}\right]^{p-1}\left(\mean{Q_{8R}}\varphi^{*}(|A(D\psi)-(A(D\psi))_{Q_{8R}}|)\,dx\right)^{\frac{2-p}{p'}} \nonumber \\
&\le c\left[\frac{|\mu|(Q_{8R})}{(8R)^{n-1}}\right] + c\left(\mean{Q_{8R}}(|Du|+s)^{(p-1)\kappa}\,dx\right)^{\frac{1}{\kappa}} \nonumber \\
& \qquad + c\left(\mean{Q_{8R}}\varphi^{*}(|A(D\psi)-(A(D\psi))_{Q_{8R}}|)\,dx\right)^{\frac{1}{p'}} \nonumber \\
\overset{\eqref{lambda.bound}}&{\le} c\left[\frac{|\mu|(Q_{4MR})}{(4MR)^{n-1}}\right] + c\left(\mean{Q_{4MR}}\varphi^{*}(|A(D\psi)-(A(D\psi))_{Q_{4MR}}|)\,dx\right)^{\frac{1}{p'}}
\end{align}
and
\begin{align}\label{wv}
\lefteqn{ \left(\mean{Q_{R/M}}|A(Dw_{1})-A(Dv)|^{\kappa}\,dx\right)^{\frac{1}{\kappa}} } \nonumber \\
\overset{\eqref{ADw1.ADv}}&{\le} c\left(\mean{Q_{4R}}|A(Dw_{1})|^{\kappa}\,dx\right)^{\frac{1}{\kappa}} + c\left(\mean{Q_{4R}}\varphi^{*}(|A(D\psi)-(A(D\psi))_{Q_{4MR}}|)\,dx\right)^{\frac{1}{p'}} \nonumber \\
& \le c\left(\mean{Q_{4MR}}|A(Du)|^{\kappa}\,dx\right)^{\frac{1}{\kappa}} + c\left(\mean{Q_{4R}}|A(Du)-A(Dw_{1})|^{\kappa}\,dx\right)^{\frac{1}{\kappa}} \nonumber \\
& \quad + c\left(\mean{Q_{4MR}}\varphi^{*}(|A(D\psi)-(A(D\psi))_{Q_{4MR}}|)\,dx\right)^{\frac{1}{p'}} \nonumber \\
\overset{\eqref{lambda.bound},\eqref{uw}}&{\le} c\left[\frac{|\mu|(Q_{4MR})}{(4MR)^{n-1}}\right] + c\left(\mean{Q_{4MR}}\varphi^{*}(|A(D\psi)-(A(D\psi))_{Q_{4MR}}|)\,dx\right)^{\frac{1}{p'}}.
\end{align}
Combining \eqref{uw} and \eqref{wv} gives \eqref{lincomp.homo.ndeg}, and the proof is complete.

\subsection{Combining the two alternatives}
Combining \lemref{a.comp.deg} and \lemref{a.comp.ndeg}, we conclude with the following comparison estimate.
\begin{lemma}\label{lin.comp.u.v}
Let $u$ and $v$ be the weak solutions to \eqref{opmu} and \eqref{homoeq}, respectively, under assumptions \eqref{growth} and \eqref{p.bound}. Then we have
\begin{align*}
\lefteqn{ \left(\mean{Q_{R/M}}|A(Du)-A(Dv)|^{\kappa}\,dx\right)^{\frac{1}{\kappa}} } \nonumber \\
& \le \varepsilon\left(\mean{Q_{4MR}}|A(Du)-\mathcal{P}_{\kappa,Q_{4MR}}(A(Du))
|^{\kappa}\,dx\right)^{\frac{1}{\kappa}} \nonumber \\
& \quad + c_{\varepsilon}\left[\frac{|\mu|(Q_{4MR})}{(4MR)^{n-1}}\right]  + c_{\varepsilon}\left(\mean{Q_{4MR}}\varphi^{*}(|A(D\psi)-(A(D\psi))_{Q_{4MR}}|)\,dx\right)^{\frac{1}{p'}}
\end{align*}
for any $\varepsilon \in (0,1)$, where $c_{\varepsilon} \equiv c_{\varepsilon} (\data,\varepsilon)$ is proportional to some negative power of $\varepsilon$.
\end{lemma}

\section{Proof of \thmref{pointwise.est} and  \thmref{mainthm.1}}\label{sec.pf.thm1}
\subsection{Excess decay estimates for $OP(\psi;\mu)$}
In \secref{sec.comparison.est} and \secref{sec.lin.comp} above, we assumed \eqref{regular} and obtained comparison estimates for weak solutions to \eqref{opmu}. 
In this section, we first obtain an excess decay estimate for weak solutions to \eqref{opmu}.
Note that we have chosen the constant $M$ depending only on $\data$ in the previous section.
\begin{lemma}
Let $u \in \mathcal{A}^{g}_{\psi}(\Omega)$ be the weak solution to \eqref{opmu} under assumptions \eqref{growth} and \eqref{p.bound}. Then
\begin{align}\label{lin.ed.est}
\lefteqn{ \left(\mean{B_{\rho}}|A(Du)-\mathcal{P}_{\kappa,B_{\rho}}(A(Du))|^{\kappa}\,dx\right)^{\frac{1}{\kappa}} } \nonumber \\
&\le \cex\left(\frac{\rho}{r}\right)^{\alpha_{A}}\left(\mean{B_{r}}|A(Du)-\mathcal{P}_{\kappa,B_{r}}(A(Du))|^{\kappa}\,dx\right)^{\frac{1}{\kappa}} \nonumber \\
& \quad + c\left(\frac{r}{\rho}\right)^{n+\gamma}\left[\frac{|\mu|(B_{r})}{r^{n-1}}\right]
 + c\left(\frac{r}{\rho}\right)^{n+\gamma}\left(\mean{B_{r}}\varphi^{*}(|A(D\psi)-(A(D\psi))_{B_{r}}|)\,dx\right)^{\frac{1}{p'}}
\end{align}
holds whenever $B_{\rho} \subset B_{r} \subset \Omega$ are concentric balls, where $c, \cex \ge 1$ and $\gamma \ge 0$ depend only on $\data$, $\kappa$ is as in \eqref{def.kappa} and $\alpha_{A} \in (0,1)$ is the exponent determined in \lemref{ed.ADv}.
\end{lemma}
\begin{proof}
Without loss of generality, we may assume that $\rho \le r/(4\sqrt{n}M^2)$. With the comparison map $v$ as in \eqref{homoeq} with $R = r/(4\sqrt{n}M)$, we apply \lemref{ed.ADv} to find
\begin{align*}
\lefteqn{ \mean{B_{\rho}}|A(Du)-\mathcal{P}_{\kappa,B_{\rho}}(A(Du))|^{\kappa}\,dx \le c\mean{Q_{\rho}}|A(Du)-\mathcal{P}_{\kappa,Q_{\rho}}(A(Dv))|^{\kappa}\,dx } \\
& \le c\mean{Q_{\rho}}|A(Dv)-\mathcal{P}_{\kappa,Q_{\rho}}(A(Dv))|^{\kappa}\,dx + c\mean{Q_{\rho}}|A(Du)-A(Dv)|^{\kappa}\,dx \\
& \le c\left(\frac{\rho}{r}\right)^{\alpha_{A}}\mean{Q_{r/(4\sqrt{n}M^2)}}\left|A(Dv)-\mathcal{P}_{\kappa,Q_{r/(4\sqrt{n}M^2)}}(A(Dv))\right|^{\kappa}\,dx \\
& \quad + c\left(\frac{r}{\rho}\right)^{n}\mean{Q_{r/(4\sqrt{n}M^2)}}|A(Du)-A(Dv)|^{\kappa}\,dx \\
& \le c\left(\frac{\rho}{r}\right)^{\alpha_{A}}\mean{Q_{r/(4\sqrt{n}M^2)}}\left|A(Du)-\mathcal{P}_{\kappa,Q_{r/(4\sqrt{n}M^2)}}(A(Du))\right|^{\kappa}\,dx \\
& \quad + c\left(\frac{r}{\rho}\right)^{n} \mean{Q_{r/(4\sqrt{n}M^2)}}|A(Du)-A(Dv)|^{\kappa}\,dx.
\end{align*}
Applying \lemref{lin.comp.u.v} to the last integral with the choice $\varepsilon = (\rho/r)^{\alpha_{A}}$ and then making elementary manipulations, we get the desired estimate.
\end{proof}

To proceed further, we now consider any limit of approximating solutions $u \in \mathcal{T}^{1,p}_{g}(\Omega)$ to $OP(\psi;\mu)$ with $\mu \in \mathcal{M}_{b}(\Omega)$. Then there exist a sequence of functions $\{\mu_k\} \subset W^{-1,p'}(\Omega) \cap L^{1}(\Omega)$ and corresponding sequence of weak solutions $\{u_k\} \subset \mathcal{A}^{g}_{\psi}(\Omega)$ to \eqref{opmu} described in \defref{def.sol}. Then the convergence properties \eqref{muk.conv} and \eqref{uk.conv} imply that \eqref{lin.ed.est} holds for $u$ as well.

\begin{lemma}\label{lin.ed.u.1-sola}
Let $u \in \mathcal{T}^{1,p}_{g}(\Omega)$ be a limit of approximating solutions to $OP(\psi;\mu)$ under assumptions \eqref{growth} and \eqref{p.bound}.
Then \eqref{lin.ed.est} still holds whenever $B_{\rho} \subset B_{r} \subset \Omega$ are concentric balls.
\end{lemma}

We now prove our main results. It suffices to prove \thmref{mainthm.1}, which with \eqref{av.min} easily implies \thmref{pointwise.est}.
\subsection{Proof of \thmref{mainthm.1}}

We start by fixing a ball $B_{2 R} \equiv B_{2 R}(x_{0}) \subset \Omega$ as in the statement.
In the following, all the balls considered will be centered at $x_{0}$.

We choose an integer $K \equiv K(\data) \geq 4M$ such that
\begin{equation*}
\frac{\cex}{K^{\alpha_{A}}} \le \frac{1}{2}.
\end{equation*}
Applying \lemref{lin.ed.u.1-sola} on arbitrary balls $B_{\rho} = B_{r/K} \subset B_{r} \Subset \Omega$, we have
\begin{align}\label{EE}
\lefteqn{ \left(\mean{B_{r/K}}|A(Du)-\mathcal{P}_{\kappa,B_{r/K}}(A(Du))|^{\kappa}\,dx\right)^{\frac{1}{\kappa}} } \nonumber \\
& \le \frac{1}{2}\left(\mean{B_{r}}|A(Du)-\mathcal{P}_{\kappa,B_{r}}(A(Du))|^{\kappa}\,dx\right)^{\frac{1}{\kappa}} \nonumber \\
& \quad + c \left[\frac{|\mu|(B_{r})}{r^{n-1}}\right] + c \left(\mean{B_{r}}\varphi^{*}(|A(D\psi)-(A(D\psi))_{B_{r}}|)\,dx\right)^{\frac{1}{p'}}.
\end{align}
For $i = 0,1,2,\ldots$, we define $R_{i} \coloneqq R/K^{i}$, $B_{i} \coloneqq B_{R_{i}}(x_{0})$, 
\begin{equation*}
\quad k_{i} \coloneqq \mathcal{P}_{\kappa,B_{i}}(A(Du)) \quad \text{and} \quad E_{i} \coloneqq \left(\mean{B_{i}}|A(Du)-\mathcal{P}_{\kappa,B_{i}}(A(Du))|^{\kappa}\,dx\right)^{\frac{1}{\kappa}}.
\end{equation*}

\textit{Step 1: Proof of \eqref{vmo.x0}.}
Applying \eqref{EE} with $r \equiv R_{i-1}$ for any $i \ge 1$, we obtain
\begin{equation}\label{EE2}
E_{i} \le \frac{1}{2}E_{i-1} + c\left[\frac{|\mu|(B_{i-1})}{R_{i-1}^{n-1}} + \left(\mean{B_{i-1}}\varphi^{*}(|A(D\psi)-(A(D\psi))_{B_{i-1}}|)\,dx\right)^{\frac{1}{p'}}\right].
\end{equation}
Iterating the above inequality, we have for any $k \geq 0$
\begin{align*}
E_{k}
& \leq 
\frac{1}{2^{k}} E_{0}
+ c \sum_{i=1}^{k} \frac{1}{2^{k-i}} \left[\frac{|\mu|(B_{i-1})}{R_{i-1}^{n-1}}
+ \left( \mean{B_{i-1}} \varphi^{*}(|A(D\psi)-(A(D\psi))_{B_{i-1}}|) \,dx \right)^{\frac{1}{p'}} \right] \\
& \leq
\frac{1}{2^{k}} E_{0}
+ c \sup_{0 < \rho \leq R} \left[\frac{|\mu|(B_{\rho})}{\rho^{n-1}} + \left( \mean{B_{\rho}} \varphi^{*}(|A(D\psi)-(A(D\psi))_{B_{\rho}}|) \,dx \right)^{\frac{1}{p'}} \right].
\end{align*}
From \eqref{mainthm.1.asmp1}, for any $\delta>0$, we temporarily fix the radius $R \equiv R(\delta)>0$ in this step to satisfy
\[ \sup_{0 < \rho \leq R} \left[\frac{|\mu|(B_{\rho})}{\rho^{n-1}} + \left( \mean{B_{\rho}} \varphi^{*}(|A(D\psi)-(A(D\psi))_{B_{\rho}}|) \,dx \right)^{\frac{1}{p'}} \right] < \delta. \]
We then choose $k_0 \in \mathbb{N}$ so large that
\[ \frac{1}{2^{k_0}} E_0 \leq \delta. \]
Consequently, for any $0 < r \leq R_{k_0}$, we obtain
\begin{align*}
\lefteqn{ \left(\mean{B_{r}} |A(Du) - \mathcal{P}_{\kappa,B_{r}}(A(Du))|^{\kappa} \, dx\right)^{\frac{1}{\kappa}} } \\
& \leq \frac{K^{n}}{2^{k_0 - 1}}E_0 + c \sup_{0 < \rho \leq R} \left[\frac{|\mu|(B_{\rho})}{\rho^{n-1}}
+ \left( \mean{B_{\rho}} \varphi^{*}(|A(D\psi)-(A(D\psi))_{B_{\rho}}|) \,dx \right)^{\frac{1}{p'}}\right] \\
& \leq c \delta.
\end{align*}
Since $\delta>0$ was arbitrary, \eqref{vmo.x0} follows.

\textit{Step 2: Proof of \eqref{Lebesgue.pt} and \eqref{mainest.1}.}
Let us first show \eqref{Lebesgue.pt}. 
Taking any $m_1<m_2 \in \mathbb{N}$ and then summing up \eqref{EE2} over $i \in \{m_{1}+1,\ldots,m_{2}\}$, we have
\begin{equation*}
\sum_{i=m_{1}+1}^{m_2}E_{i} \le \frac{1}{2}\sum_{i=m_{1}}^{m_2-1}E_{i} + c\sum_{i=m_{1}}^{m_2-1}\left[\frac{|\mu|(B_{i})}{R_{i}^{n-1}} + \left(\mean{B_{i}}\varphi^{*}(|A(D\psi)-(A(D\psi))_{B_{i}}|)\,dx\right)^{\frac{1}{p'}} \right]
\end{equation*}
and hence
\begin{equation}\label{sumEi}
\sum_{i=m_{1}}^{m_2}E_{i} \le 2 E_{m_{1}} + 2 c\sum_{i=m_{1}}^{m_2-1}\left[\frac{|\mu|(B_{i})}{R_{i}^{n-1}} + \left(\mean{B_{i}}\varphi^{*}(|A(D\psi)-(A(D\psi))_{B_{i}}|)\,dx\right)^{\frac{1}{p'}} \right].
\end{equation}

We observe the following elementary inequalities (see for instance \cite[(115)]{KM14BMS}):
\begin{equation}\label{sum.int.mu}
\sum_{i=m_{1}}^{m_{2}-1}\left[\frac{|\mu|(B_{i})}{R_{i}^{n-1}}\right] \le c(K)\mathbf{I}^{\mu}_{1}(x_{0},2 R_{m_{1}})
\end{equation}
and
\begin{align}\label{sum.int.psi}
\lefteqn{ \sum_{i=m_{1}}^{m_{2}-1}\left(\mean{B_{i}}\varphi^{*}(|A(D\psi)-(A(D\psi))_{B_{i}}|)\,dx\right)^{\frac{1}{p'}} } \nonumber \\ 
&\le c(K)\int_{0}^{2 R_{m_{1}}}\left(\mean{B_{\rho}(x_{0})}\varphi^{*}(|A(D\psi)-(A(D\psi))_{B_{\rho}(x_{0})}|)\,dx\right)^{\frac{1}{p'}}\frac{d\rho}{\rho}.
\end{align}
Plugging \eqref{sum.int.mu} and \eqref{sum.int.psi} into \eqref{sumEi}, we have
\begin{align}\label{sumEi.2}
\lefteqn{ |k_{m_{1}} - k_{m_{2}}|
 \leq \sum_{i=m_{1}}^{m_{2}-1}|k_{i}-k_{i+1}| \leq cK^{\frac{n}{\kappa}} \sum_{i = m_{1}}^{m_{2}-1} E_i } \nonumber \\
& \leq c E_{m_{1}} + c \mathbf{I}_{1}^{\mu}(x_0, 2R_{m_{1}}) + c \int_{0}^{2R_{m_{1}}} \left( \mean{B_{\rho}} \varphi^{*} (|A(D\psi) - (A(D\psi))_{B_{\rho}}|) \dx \right)^{\frac{1}{p'}} \frac{d\rho}{\rho}.
\end{align}
Note that \eqref{mainthm.1.asmp2} implies \eqref{mainthm.1.asmp1} and
\[ \lim_{r \to 0} \left[ \mathbf{I}_{1}^{\mu}(x_0,r) + \int_{0}^{r} \left( \mean{B_{\rho}} \varphi^{*} (|A(D\psi) - (A(D\psi))_{B_{\rho}}|) \dx \right)^{\frac{1}{p'}} \frac{d\rho}{\rho} \right] = 0. \]
In particular, as a consequence of Step 1, we have \eqref{vmo.x0}.
Accordingly, for every $\varepsilon>0$, we can take $N \in \mathbb{N}$ such that
\[ E_{N} + \mathbf{I}_{1}^{\mu}(x_0, 2R_{N}) + \int_{0}^{2R_{N}} \left( \mean{B_{\rho}} \varphi^{*} (|A(D\psi) - (A(D\psi))_{B_{\rho}}|) \dx \right)^{\frac{1}{p'}} \frac{d\rho}{\rho} < \varepsilon. \]
From this and \eqref{sumEi.2}, we see that 
\[ |k_{m_{1}} - k_{m_{2}}| < c \varepsilon \quad \text{whenever }N \leq m_1 < m_2, \]
which implies that $\{ k_{i} \}$ is a Cauchy sequence in $\mathbb{R}^n$.
We therefore obtain \eqref{Lebesgue.pt}.

Now, in order to show \eqref{mainest.1}, we again take an arbitrary small constant $\varepsilon>0$.
In light of \eqref{Lebesgue.pt}, we can take $m \in \mathbb{N}$ large enough to satisfy
\[ |A_{0} - \mathcal{P}_{\kappa,B_m}(A(Du))| \leq \varepsilon. \]
It then follows from \eqref{sumEi.2} that
\begin{align}\label{pt.bound}
\lefteqn{ |A_{0} - \mathcal{P}_{\kappa,B_0}(A(Du))|  \leq |A_{0} - \mathcal{P}_{\kappa,B_m}(A(Du))| + |\mathcal{P}_{\kappa,B_m}(A(Du)) - \mathcal{P}_{\kappa,B_0}(A(Du))| } \nonumber \\
& \leq \varepsilon + c E_{0} + c \mathbf{I}_{1}^{\mu}(x_0, 2R)  + c \int_{0}^{2R} \left( \mean{B_{\rho}(x_{0})} \varphi^{*} (|A(D\psi) - (A(D\psi))_{B_{\rho}(x_{0})}|) \dx \right)^{\frac{1}{p'}} \frac{d\rho}{\rho}.
\end{align}
Recalling that $\varepsilon$ is arbitrary, we obtain \eqref{mainest.1} as follows:
\begin{align*}
|A_{0} - \mathcal{P}_{\kappa,B_{2R}(x_0)}(A(Du))| 
& \leq |A_{0} - \mathcal{P}_{\kappa,B_0}(A(Du))| + |\mathcal{P}_{\kappa,B_0}(A(Du)) - \mathcal{P}_{\kappa,B_{2R}(x_0)}(A(Du))| \\
\overset{\eqref{pt.bound}}&{\leq} c\left(\mean{B_{2R}(x_0)}|A(Du) - \mathcal{P}_{\kappa,B_{2R}(x_0)}(A(Du))|^{\kappa} \,dx\right)^{\frac{1}{\kappa}} + c \mathbf{I}_{1}^{\mu}(x_0, 2R) \\
& \quad + c \int_{0}^{2R} \left( \mean{B_{\rho}(x_{0})} \varphi^{*} (|A(D\psi) - (A(D\psi))_{B_{\rho}(x_{0})}|) \,dx \right)^{\frac{1}{p'}} \frac{d\rho}{\rho}.
\end{align*}
Finally, if $x_{0}$ is a Lebesgue point of $A(Du)$, then \eqref{av.min} implies
\begin{align*}
|A(Du(x_{0}))-\mathcal{P}_{\kappa,B_{\rho}(x_{0})}(A(Du))| & \le c\left(\mean{B_{\rho}(x_{0})}|A(Du)-A(Du(x_{0}))|^{\kappa}\,dx\right)^{\frac{1}{\kappa}} \\
& \le c\mean{B_{\rho}(x_{0})}|A(Du)-A(Du(x_{0}))|\,dx.
\end{align*}
Hence, letting $\rho\rightarrow0$, the last assertion in \thmref{mainthm.1} follows.
\, \hfill \qed

\

\noindent \textbf{Conflict of interest.} The authors declare that they have no conflict of interest.

\noindent \textbf{Data availability.} Data sharing not applicable to this article as no datasets were generated or analyzed during the current study.

\providecommand{\bysame}{\leavevmode\hbox to3em{\hrulefill}\thinspace}
\providecommand{\MR}{\relax\ifhmode\unskip\space\fi MR }
\providecommand{\MRhref}[2]{%
  \href{http://www.ams.org/mathscinet-getitem?mr=#1}{#2}
}
\providecommand{\href}[2]{#2}


\begin{thebibliography}{10}

\bibitem{AKM18}
B.~Avelin, T.~Kuusi, and G.~Mingione, \emph{Nonlinear {C}alder\'{o}n-{Z}ygmund theory in the limiting case}, Arch. Ration. Mech. Anal. \textbf{227} (2018), no.~2, 663--714.

\bibitem{BDW20}
A.~Kh.~Balci, L.~Diening, and M.~Weimar, \emph{Higher order
  {C}alder\'{o}n-{Z}ygmund estimates for the {$p$}-{L}aplace equation}, J. Differential Equations \textbf{268} (2020), no.~2, 590--635.

\bibitem{Ba15}
P.~Baroni, \emph{Riesz potential estimates for a general class of quasilinear equations}, Calc. Var. Partial Differential Equations \textbf{53} (2015), no.~3-4, 803--846.

\bibitem{BaHa14}
P.~Baroni and J.~Habermann, \emph{Elliptic interpolation estimates for non-standard growth operators}, Ann. Acad. Sci. Fenn. Math. \textbf{39} (1) (2014), 119--162.

\bibitem{BBGGPV1995}
P.~B\'enilan, L.~Boccardo, T.~Gallou\"et, R.~Gariepy, M.~Pierre, and J.~L.~V\'azquez ,
\emph{An $L^1$-theory of existence and uniqueness of solutions of nonlinear elliptic equations}, Ann. Scuola Norm. Sup. Pisa Cl. Sci. (4) \textbf{22} (1995), 241--273.

\bibitem{BG89}
L.~Boccardo and T.~Gallou\"{e}t, \emph{Nonlinear elliptic and parabolic equations involving measure data}, J. Funct. Anal. \textbf{87} (1989), no.~1, 149--169.

\bibitem{BCDKS18}
D.~Breit, A.~Cianchi, L.~Diening, T.~Kuusi, and S.~Schwarzacher, \emph{Pointwise {C}alder\'{o}n-{Z}ygmund gradient estimates for the {$p$}-{L}aplace system}, J. Math. Pures Appl. (9) \textbf{114} (2018), 146--190.

\bibitem{BCP21}
S.-S. Byun, Y. Cho, and J.-T. Park, \emph{Nonlinear gradient estimates for elliptic double obstacle problems with measure data}, J. Differential Equations \textbf{293} (2021), 249--281.

\bibitem{BSY}
S.-S. Byun, K. Song, and Y. Youn, \emph{Potential estimates for elliptic measure data problems with irregular obstacles}, Math. Ann., to appear. https://doi.org/10.1007/s00208-022-02471-z

\bibitem{BSY2}
S.-S. Byun, K. Song, and Y. Youn, \emph{Fractional differentiability for elliptic double obstacle problems with measure data}, Z. Anal. Anwend., to appear. https://doi.org/10.4171/zaa/1721

\bibitem{BY17}
S.-S.~Byun and Y.~Youn, \emph{Optimal gradient estimates via {R}iesz potentials for {$p(\cdot)$}-{L}aplacian type equations}, Q. J. Math. \textbf{68} (2017), no.~4, 1071--1115.

\bibitem{BY18}
S.-S.~Byun and Y.~Youn, \emph{Riesz potential estimates for a class of double phase problems}, J. Differential Equations \textbf{264} (2018), no.~2, 1263--1316.

\bibitem{BY19}
S.-S.~Byun and Y.~Youn, \emph{Potential estimates for elliptic systems with subquadratic growth}, J. Math. Pures Appl. (9) \textbf{131} (2019), 193--224.

\bibitem{CM17NA}
A.~Cianchi and V. Maz'ya, \emph{Quasilinear elliptic problems with general growth and merely integrable, or measure, data}, Nonlinear Anal. \textbf{164} (2017), 189--215.

\bibitem{CS18}
A.~Cianchi and S.~Schwarzacher, \emph{Potential estimates for the {$p$}-{L}aplace system with data in divergence form}, J. Differential Equations \textbf{265} (2018), no.~1, 478--499.

\bibitem{DMOP99}
G. Dal~Maso, F. Murat, L. Orsina, and A. Prignet, \emph{Renormalized solutions of elliptic equations with general measure data}, Ann. Sc. Norm. Super. Pisa Cl. Sci. (4) \textbf{28} (1999), no.~4, 741--808. 

\bibitem{D22}
C. De~Filippis, \emph{Quasiconvexity and partial regularity via nonlinear potentials}, J. Math. Pures Appl. (9) \textbf{163} (2022), 11--82.

\bibitem{DS23JFA}
C. De~Filippis and B. Stroffolini, \emph{Singular multiple integrals and nonlinear potentials}, J. Funct. Anal. \textbf{285} (2023), no.~2, Paper No. 109952.   

\bibitem{DJL92}
R.~A. DeVore, B. Jawerth, and B.~J. Lucier, \emph{Image compression through wavelet transform coding}, IEEE Trans. Inform. Theory \textbf{38} (1992), no.~2, part 2, 719--746.

\bibitem{DVS84}
R.~A. DeVore and R.~C. Sharpley, \emph{Maximal functions measuring smoothness}, Mem. Amer. Math. Soc. \textbf{47} (1984), no.~293, viii+115.

\bibitem{DFTW}
L.~Diening, M.~Fornasier, R.~Tomasi, and M.~Wank, \emph{A relaxed {K}a\v{c}anov iteration for the {$p$}-{P}oisson problem}, Numer. Math. \textbf{145} (2020), no.~1, 1--34. 

\bibitem{DKS12}
L.~Diening, P.~Kaplick\'{y}, and S.~Schwarzacher, \emph{B{MO} estimates for the {$p$}-{L}aplacian}, Nonlinear Anal. \textbf{75} (2012), no.~2, 637--650.

\bibitem{DZ}
H.~Dong and H.~Zhu, \emph{Gradient estimates for singular $p$-Laplace type equations with measure data}, Preprint (2021), arXiv:2102.08584.

\bibitem{DZ22}
H. Dong and H. Zhu, \emph{Gradient estimates for singular parabolic {$p$}-{L}aplace type equations with measure data}, Calc. Var. Partial Differential Equations \textbf{61} (2022), no.~3, Paper No. 86, 41.

\bibitem{DM10CV}
F.~Duzaar and G.~Mingione, \emph{Gradient continuity estimates}, Calc. Var. Partial Differential Equations \textbf{39} (2010), no.~3-4, 379--418.

\bibitem{DM10JFA}
F.~Duzaar and G.~Mingione, \emph{Gradient estimates via linear and nonlinear potentials}, J. Funct. Anal. \textbf{259} (2010), no.~11, 2961--2998.

\bibitem{DM11AJM}
F.~Duzaar and G.~Mingione, \emph{Gradient estimates via non-linear potentials}, Amer. J. Math. \textbf{133} (2011), no.~4, 1093--1149.

\bibitem{KM92}
T.~Kilpel\"{a}inen and J.~Mal\'{y}, \emph{Degenerate elliptic equations with measure data and nonlinear potentials}, Ann. Sc. Norm. Super. Pisa Cl. Sci. (4) \textbf{19} (1992), no.~4, 591--613.

\bibitem{KM94}
T.~Kilpel\"{a}inen and J.~Mal\'{y}, \emph{The {W}iener test and potential estimates for quasilinear elliptic equations}, Acta Math. \textbf{172} (1994), no.~1, 137--161.

\bibitem{KS00}
D.~Kinderlehrer and G.~Stampacchia, \emph{An introduction to variational inequalities and their applications}, Classics in Applied Mathematics, vol.~31, Society for Industrial and Applied Mathematics (SIAM), Philadelphia, PA, 2000, Reprint of the 1980 original.

\bibitem{KM12JFA}
T.~Kuusi and G.~Mingione, \emph{Universal potential estimates}, J. Funct. Anal. \textbf{262} (2012), no.~10, 4205--4269.

\bibitem{KM13Pisa}
T.~Kuusi and G.~Mingione, \emph{Gradient regularity for nonlinear parabolic equations}, Ann. Sc. Norm. Super. Pisa Cl. Sci. (5) \textbf{12} (2013), no.~4, 755--822.

\bibitem{KM13ARMA}
T.~Kuusi and G.~Mingione, \emph{Linear potentials in nonlinear potential theory}, Arch. Ration. Mech. Anal. \textbf{207} (2013), no.~1, 215--246.

\bibitem{KM14BMS}
T.~Kuusi and G.~Mingione, \emph{Guide to nonlinear potential estimates}, Bull. Math. Sci. \textbf{4} (2014), no.~1, 1--82.

\bibitem{KM14CV}
T.~Kuusi and G.~Mingione, \emph{A nonlinear {S}tein theorem}, Calc. Var. Partial Differential
  Equations \textbf{51} (2014), no.~1-2, 45--86.

\bibitem{KM14ARMA}
T.~Kuusi and G.~Mingione, \emph{Riesz potentials and nonlinear parabolic equations}, Arch. Ration. Mech. Anal. \textbf{212} (2014), no.~3, 727--780.

\bibitem{KM16JEP}
T.~Kuusi and G.~Mingione, \emph{Partial regularity and potentials}, J. \'{E}c. polytech. Math. \textbf{3} (2016), 309--363.

\bibitem{KM18JEMS}
T.~Kuusi and G.~Mingione, \emph{Vectorial nonlinear potential theory}, J. Eur. Math. Soc. (JEMS) \textbf{20} (2018), no.~4, 929--1004.

\bibitem{Min07}
G.~Mingione, \emph{The {C}alder\'{o}n-{Z}ygmund theory for elliptic problems with measure data}, Ann. Sc. Norm. Super. Pisa Cl. Sci. (5) \textbf{6} (2007), no.~2, 195--261.

\bibitem{Min11JEMS}
G.~Mingione, \emph{Gradient potential estimates}, J. Eur. Math. Soc. (JEMS) \textbf{13} (2011), no.~2, 459--486.

\bibitem{Min11Milan}
G.~Mingione, \emph{Nonlinear measure data problems}, Milan J. Math. \textbf{79} (2011), no.~2, 429--496.

\bibitem{NP20JFA}
Q.-H.~Nguyen and N.~C.~Phuc, \emph{Pointwise gradient estimates for a class of singular quasilinear equations with measure data}, J. Funct. Anal. \textbf{278} (2020), no.~5, 108391, 35.

\bibitem{NP23ME}
Q.-H. Nguyen and N.~C. Phuc, \emph{Universal potential estimates for {$1 < p \leq 2 -\frac 1n$}}, Math. Eng. \textbf{5} (2023), no.~3, Paper No. 057, 24.

\bibitem{NP23ARMA}
Q.-H.~Nguyen and N.~C.~Phuc, \emph{A comparison estimate for singular {$p$}-{L}aplace equations and its consequences}, Arch. Ration. Mech. Anal. \textbf{247} (2023), no.~3, 49. https://doi.org/10.1007/s00205-023-01884-7

\bibitem{PS22}
J.-T. Park and P. Shin, \emph{Regularity estimates for singular parabolic measure data problems with sharp growth}, J. Differential Equations \textbf{316} (2022), 726--761. 

\bibitem{Sch12PM}
C.~Scheven, \emph{Elliptic obstacle problems with measure data: potentials and low order regularity}, Publ. Mat. \textbf{56} (2012), no.~2, 327--374.

\bibitem{Sch12JFA}
C.~Scheven, \emph{Gradient potential estimates in non-linear elliptic obstacle problems with measure data}, J. Funct. Anal. \textbf{262} (2012), no.~6, 2777--2832.

\bibitem{SY}
K. Song and Y. Youn, \emph{A note on comparison principle for elliptic obstacle problems with $L^{1}$-data}, Bull. Korean Math. Soc. \textbf{60} (2023), no.~2, 495--505.

\bibitem{TW02}
N.~S.~Trudinger and X.-J.~Wang, \emph{On the weak continuity of elliptic operators and applications to potential theory}, Amer. J. Math. \textbf{124} (2002), no.~2, 369--410.

\end{thebibliography}
\end{document}